\documentclass{amsproc}

\usepackage{amsmath}
\usepackage{amssymb}
\usepackage{amsthm}
\usepackage{cases}
\usepackage{graphicx}
\usepackage{upgreek}

\usepackage{xcolor}
\definecolor{MyGrey}{rgb}{.804,.804,.756}
\usepackage[colorlinks=false,linkbordercolor=MyGrey,citebordercolor=MyGrey,
urlbordercolor=MyGrey, pdfauthor={Carter, Lebed, Yang}, pdftitle={A prismatic classifying space}]{hyperref}

\newtheorem{theorem}{Theorem}[section]
\newtheorem{lemma}[theorem]{Lemma}
\theoremstyle{definition}

\theoremstyle{remark}

\numberwithin{equation}{section}

\newcommand*\lt{\mathrel{\triangleleft}}

\def\k{{\it k}}

\def\End{{\rm End}}

\newcommand{\N}{\mathbb{N}}

\newcommand{\Z}{\mathbb{Z}}
\newcommand{\R}{\mathbb{R}}

\newcommand{\X}{{\mbox{\sf X}}}

\newcommand{\Y}{{\mbox{\sf Y}}}
\newcommand{\T}{{\mbox{\sf T}}}
\newcommand{\A}{{\mbox{\sf H}}}%{{\mbox{\sf A}}}

\newcommand{\I}{{\mbox{\sf I}}}

\newcommand{\Ker}{\operatorname{Ker}}
\renewcommand{\Im}{\operatorname{Im}}

\begin{document}

\vspace*{1.5pc}\large

\title{A prismatic classifying space}

\author{J. Scott Carter}
\address{University of South Alabama}
\curraddr{Department of Mathematics and Statistics\\Mobile, AL 36688, USA}
\email{carter@southalabama.edu}
\thanks{Supported by Simons Foundation collaborative grant: 318381}

\author{Victoria Lebed}
\address{Trinity College Dublin}
\curraddr{Hamilton Mathematics Institute, Trinity College, Dublin 2, Ireland}
\email{lebed@maths.tcd.ie}
\thanks{Supported by a Hamilton Research Fellowship (Hamilton Mathematics Institute, Trinity College Dublin).}

\author{Seung Yeop Yang}
\address{University of Denver}
\curraddr{Department of Mathematics\\Denver, CO 80208, USA}
\email{seungyeop.yang@du.edu}
\thanks{}


\date{}

\begin{abstract}

A qualgebra $G$ is a set having two binary operations that satisfy compatibility conditions which are modeled upon a group under conjugation and multiplication.
We develop a homology theory for qualgebras and describe a classifying space for it. This space is constructed from $G$-colored prisms (products of simplices) and simultaneously generalizes (and includes) simplicial classifying spaces for groups and cubical classifying spaces for quandles. Degenerate cells of several types are added to the regular prismatic cells; by duality, these correspond to ``non-rigid'' Reidemeister moves and their higher dimensional analogues. Coupled with $G$-coloring techniques, our homology theory yields invariants of knotted trivalent graphs in $\mathbb{R}^3$ and knotted foams in $\mathbb{R}^4$. We re-interpret these invariants as homotopy classes of maps from $S^2$ or $S^3$ to the classifying space of $G$.

\end{abstract}

\maketitle

\section{Introduction}

Consider a group, $G$, that acts upon itself via conjugation: $a \lt b = b^{-1}a b$. The group operation $\cdot$ (or more simply juxtaposition) and the conjugation operation $\lt$ satisfy the following compatibility conditions, for all $a,b,c \in G$:
\begin{enumerate}
\item[$\A$]\quad $(ab)c=a(bc)$ \hfill (associativity);
\item[$\Y\I$] \quad $(ab)\lt c= (a\lt c)(b \lt c)$ \hfill  (distributivity);
\item[$\I\Y$]\quad  $(a\lt b) \lt c= a \lt (bc)$ \hfill  (exponential law);
\item[$\I\I\I$] \quad $(a \lt b) \lt c = (a \lt c) \lt (b\lt c)$ \hfill (self-distributivity);
\item[$\I\I$] \quad $\forall a,b$ $\exists ! c$ such that $c\lt b =a$ \hfill (right invertibility);
\item[$\I $]  \quad $a\lt a =a$ \hfill  (idempotence);
\item [$\T$] \quad $a\cdot b = b \cdot(a\lt b)$ \hfill  (twisted commutativity).
\end{enumerate}

Precisely the same rules appear when one extends quandle coloring techniques from knots to knotted trivalent graphs and handle-body knots \cite{AI:MCQ,VL:qual} (see Fig.~\ref{coloredvertices}). Each of the above axioms translates a Reidemeister move for these objects (see Fig.~\ref{Rmoves}). We chose our axiom names to evoke these moves.

\begin{figure}[htb]
\begin{center}
\includegraphics[width=0.3\paperwidth]{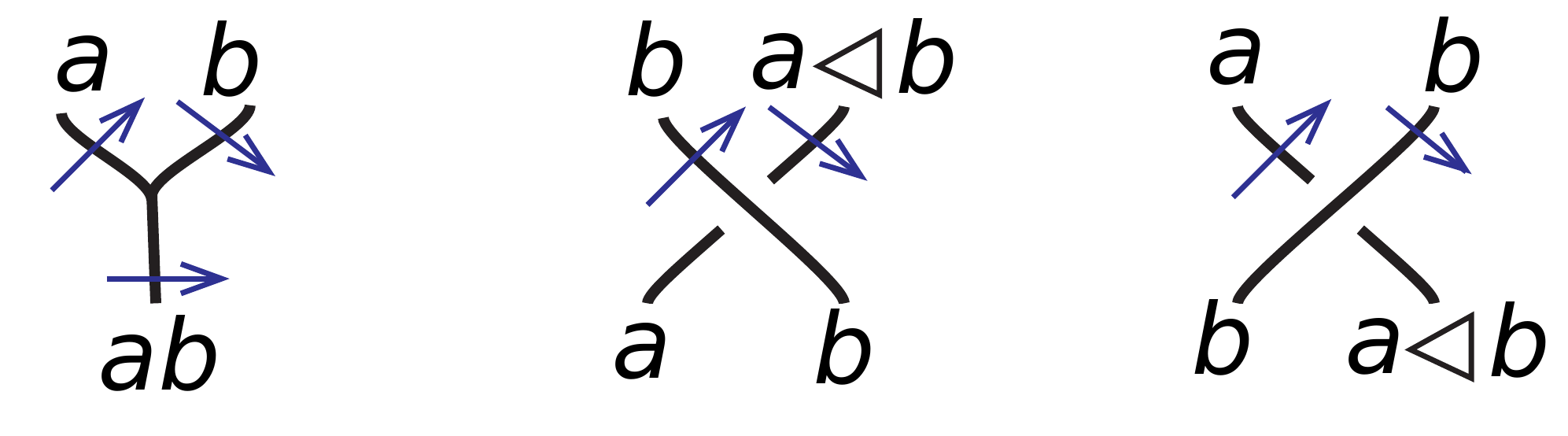}
\end{center}
\caption{Coloring rules for knotted trivalent graphs}\label{coloredvertices}
\end{figure}

\begin{figure}[htb]
\begin{center}
\includegraphics[width=0.6\paperwidth]{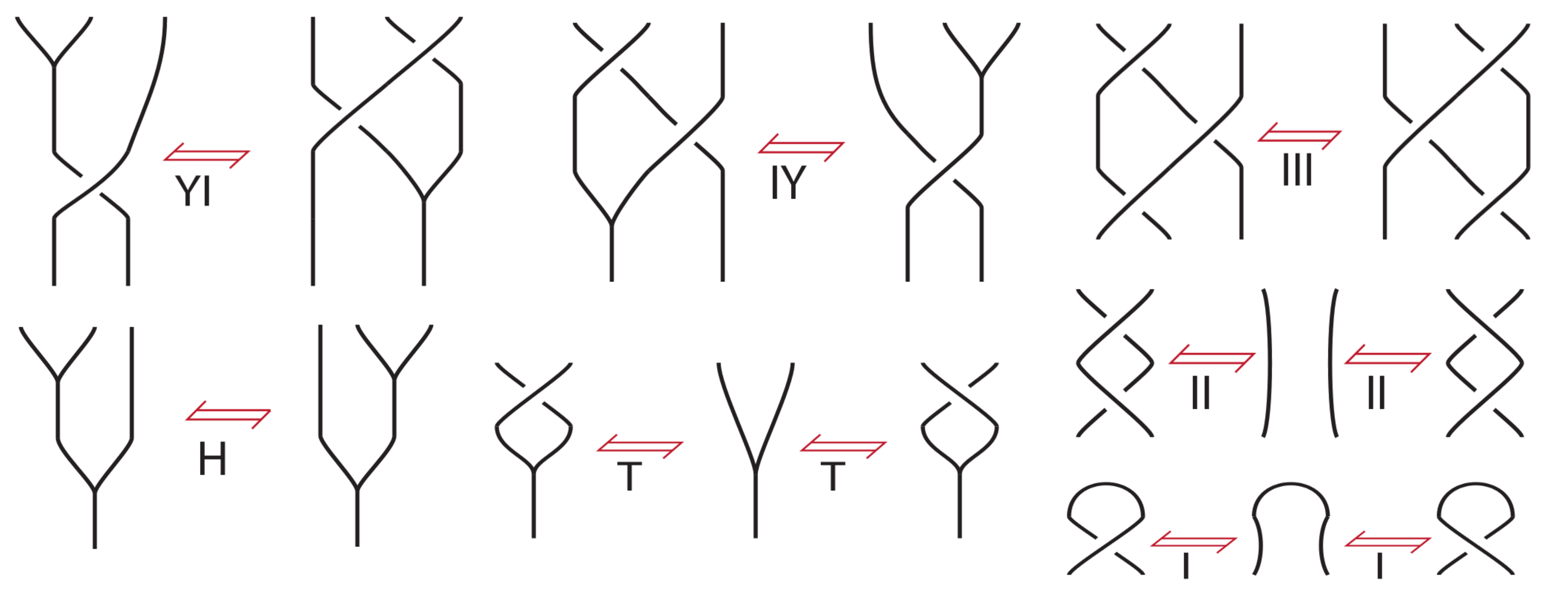}
\end{center}
\caption{Reidemeister moves for knotted trivalent graphs (without $\A$) and handle-body knots (including $\A$)}\label{Rmoves}
\end{figure}

Different subsets of our seven axioms also emerge in the study of elementary embeddings and Laver tables in set theory, and extended and parenthesized braids in braid theory. See \cite{DrapalLDM,DrapalLDM2,DehornoyParBr,DehornoyFreeALDS} and \cite[Chapter~$\mathrm{XI}$]{Dehornoy2} for more details.

We will call a set with two binary operations $\cdot$ and $\lt$ satisfying the seven axioms\footnote{In fact, Axiom $\I\I\I$ is a consequence of $\I\Y$ and $\T$. We leave it in the list to be able to select sub-lists defining different related structures.} a \emph{qualgebra}\footnote{The term is borrowed from \cite{VL:qual}, where the objects we discuss are called \emph{associative qualgebras}. Since we are interested only in the associative case, this adjective will be omitted.} (= quandle + algebra), and talk about a \emph{shalgebra} (= shelf + algebra) if only the first four axioms hold. See \cite{VL:qual} for non-group qualgebra examples. Recall that a \emph{quandle} (resp., a \emph{shelf}) is a set with a binary operation $\lt$ subject to Axioms $\I\I\I$-$\I$ (resp., $\I\I\I$ only). In addition to abstract algebra,  knot and set theories,
shelves play an important role in Hopf algebra classification, integration of Leibniz algebras, the study of the Yang--Baxter equation etc. Quandles and shelves can be thought of as algebraic counterparts of knots and positive braids respectively. In the same spirit, qualgebras and shalgebras ``algebrize'' handle-body knots and branched braids (i.e., braids with some zip- and unzip-like branching points, cf. \cite{VL:BranchedBraids}).

Many of the classical applications of quandles have enhanced versions involving quandle (co)homology. An analogous theory was described for qualgebras in small degree \cite{VL:qual}, and for multiple conjugation quandles (a particular type of qualgebras) in any degree \cite{CIST:HomMCQ}. This allows
 for enhancements of
qualgebra-based invariants for knotted trivalent graphs and handle-bodies, as well as invariants of knotted $2$-foams in $\R^4$.

In the present work, we extend the homology theory from \cite{CIST:HomMCQ} to any shalgebra $(G,\cdot,\lt)$ (Section~\ref{S:PrismHom}), and construct a cell complex $BG$ computing this homology (Section \ref{S:all}). Cells of small degree, which are the most important for knot-theoretic applications, are  given particular attention. Our cell complex interpolates between the classifying spaces of the semigroup $(G,\cdot)$ and of the shelf $(G,\lt)$, recalled in Section~\ref{S:SimplicesCubes}. Its cells are $G$-colored generalized prisms of different shape, involving simplices (inherited from associative homologies) and cubes (inherited from self-distributive homologies). By duality, our cell complex corresponds to knotted trivalent graphs and their higher-dimensional analogues, e.g. knotted foams, described in Section~\ref{S:Knottings}. The homology theory for qualgebras is obtained from that for shalgebras by taking a quotient of the defining chain complex. The classifying space $BG$ is then amended with additional cells (Section~\ref{S:Degeneracies}).
The idea of encoding additional axioms by sub-complexes is omni-present in algebra: it allows one to construct group and quandle homology theories out of those for semigroups and, respectively, shelves. For more on the sub-complex philosophy, see \cite{EK_MS:IdentitiesHomology}. The main application of our qualgebra homology theory, and in particular of its interpretation in terms of the classifying space $BG$, is the construction of invariants of handle-body knots and knotted foams presented in Section~\ref{S:Invariants}. There are several ways of thinking about these invariants. First of all, they are a natural consequence of the duality between our prisms and the knotted objects considered. Taking a step further, we interpret colored diagrams of $n$-dimensional knottings as based homotopy classes of maps from $S^{n+1}$ into $BG$. Here we work with $n=1$ and $n=2$ only, but the core of our methods adapts to any dimension.

\section{Prismatic homology}\label{S:PrismHom}

Fix a shalgebra $(G,\cdot,\lt)$. For any integer $n>0$, define an abelian group
\[C_n= \bigoplus_{k_1+ k_2 + \ldots + k_\ell=n, \; 1\leq \ell \leq n, \; k_j \geq 1} \Z[G^{k_1}\times G^{k_2} \times \cdots \times G^{k_\ell}].\]
Put $C_0 =0$. In this section we will endow the collection ${\mathcal C} =\{C_n: n\ge 0\}$ with differentials $\partial \colon C_n \to C_{n-1}$, turning it into a chain complex. Even better: the property $\partial^2=0$ will be equivalent to shalgebra axioms.

The notation $((g_1, \ldots,g_{k_1}),\ldots,(g_{s+1}, \ldots,g_{s+k_\ell}))$ for linear generators of $C_n$ will be simplified to $(g_1, \ldots,g_{k_1})| \ldots |(g_{s+1}, \ldots,g_{s+k_\ell})$; here $s=\sum_{h=1}^{{\ell-1}}k_h$. Also, if some $k_j=1$, we will write $g_{t+1}$ instead of $(g_{t+1})$; here $t=\sum_{h=1}^{{j-1}}k_h$.

The $\lt$-action of $G$ on itself extends to its powers diagonally:
\[(g_1, \ldots,g_k) \lt h = (g_1 \lt h, \ldots,g_k \lt h).\]

Now, for a generator $(g_{s+1}, \ldots, g_{s+k_j})$ of the intermediate factor $G^{k_j}$ of $C_n$, where $s=\sum_{h=1}^{{j-1}}k_h$, put
\begin{align}
\partial_{j;k_j} (g_{s+1}, \ldots, &g_{s+k_j})  = \lt g_{s+1}( g_{s+2}, \ldots, g_{s+k_{j}}) \notag\\
& + \sum_{i=1}^{k_j-1} (-1)^i ( g_{s+1}, \ldots, g_{s+i}\cdot g_{s+i+1}, \ldots, g_{s+k_j}) \label{E:PrismHom}\\
& +\ (-1)^{k_j}(g_{s+1}, \ldots, g_{s+k_j-1}). \notag
\end{align}
Here $\lt g_{s+1}$ indicates that $g_{s+1}$ acts as described above on anything to its left. 
We simplify $\partial = \partial_{j;k_j}$ when the subscript are clear from the context. Next, suppose that $P$ and $Q$ are generators of the chain groups $C_p$ and $C_q$, respectively, and that $\partial P$ and $\partial Q$ are already defined. Then use the Leibniz rule to define $\partial(P|Q)$:
\[\partial(P|Q)= (\partial P)| Q + (-1)^{p} P |(\partial Q).\]
This defines a differential $\partial \colon C_n \to C_{n-1}$ for all $n$ by induction. Indeed, the relation $\partial^2=0$ holds 
\begin{itemize}
\item on the intermediate factor $G^{k_j}$ due to the associativity of~$\cdot$ (Axiom $\A$) and the fact that the $\lt$-action is a semigroup action (Axiom $\I\Y$), as it is the case for a general bar complex with coefficients;
\item on the whole $C_n$ because the notation $(\partial P)|(\partial Q)$ is unambiguous, which is equivalent to the properties $\lt g \lt h = \lt h \lt (g \lt h)$ (Axiom $\I\I\I$) and $(g_i \cdot g_{i+1}) \lt h = (g_i \lt h) \cdot (g_{i+1} \lt h)$ (Axiom $\Y\I$).
\end{itemize}
An alternative proof of $\partial^2=0$ uses braided systems, cf. \cite{VL:Systems}.

In a moment, we will give explicit formulas for small degrees, but first note that our definition implies that $\partial_{j;1} (g_{s+1}) = (\lt g_{s+1} \text{\textvisiblespace} - \text{\textvisiblespace}).$ That is if any $k_j=1$, that particular factor will act on everything to its left and not act. The differential in that factor is a difference of the action and inaction. In particular, $\partial(g|Q)= Q-Q-g|(\partial Q)=-g|(\partial Q)$.

Furthermore, we remark that the construction of our chain complex is analogous to repeatedly taking the tensor product of the complexes for semigroup homology, with the exception that the $\lt$-action must be taken into consideration.

Let us proceed to compute some low dimensional boundaries. We chose to keep the canceling terms (the ones with allow one to reduce $\partial(g|Q)$ to $-g|(\partial Q)$) since they will appear in the geometric realization of our homology, as two copies of the same cell with opposite orientations.
\begin{eqnarray*}
\text{deg }2 : \quad \partial (a,b) &=& b-ab+a; \\
 \partial (a|b) &=& b-b -a \lt b + a;
\end{eqnarray*}
\begin{eqnarray*}
\text{deg }3 : \quad \partial (a,b,c) &=& (b,c) -(ab,c) +(a,bc)-(a,b); \\
 \partial ((a,b)|c) &=& b|c-ab|c+a|c +(a\lt c, b \lt c)-(a,b);\\
 \partial (a|(b,c)) &=& (b,c)-(b,c)-a \lt b|c + a|bc - a|b; \\
 \partial (a|b|c) &=& b|c-b|c -a\lt b|c + a|c +a\lt c|b \lt c -a|b;
\end{eqnarray*}
\begin{eqnarray*}
\text{deg }4 : \quad \partial  (a,b,c,d)  &=& ( b,c,d ) - (ab,c,d)  + (a,bc,d) - (a,b,cd) + (a,b,c); \\
\partial ((a,b,c)| d)  &=& (b,c) | d  -(ab,c)|d +  (a,bc)|d - (a,b)|d\\
&&- (a\lt d, b \lt d, c\lt d)+ (a,b,c); \\
 \partial ((a,b)|(c, d)) &=& b|(c,d) -ab|(c,d)+a|(c, d)\\
  && +(a \lt c,b \lt c )|d-(a ,b )|cd+(a,b)|c; \\
 \partial((a,b)|c|d) &=& b|c|d-ab|c|d+a|c|d +(a \lt c,b \lt c)|d-(a,b)|d\\ && - (a \lt d,b \lt d)|c \lt d + (a,b)|c; \\
 \partial (a| (b,c,d))&=& (b,c,d)-(b,c,d)\\
 &&- a \lt b|(c,d)+a|(bc,d)-a|(b,cd)+a|(b,c); \\
 \partial (a|(b,c)|d)  &=& (b,c)|d-(b,c)|d - a \lt b|c|d +a|bc|d-a|b|d\\
 &&-a \lt d|(b \lt d ,c \lt d)+a|(b,c); \\
 \partial (a|b|(c,d)) &=&b|(c,d)-b|(c,d)-a \lt b|(c, d)+a|(c,d)\\
 &&+ a \lt c| b \lt c|d-a|b|cd + a|b|c; \\
 \partial (a|b|c|d)&=&b|c|d-b|c|d -a \lt b|c|d+a|c|d\\
 &&+a \lt c | b \lt c| d - a|b|d-a \lt d | b \lt d|c \lt d+ a|b|c.
\end{eqnarray*}

The {\it prismatic homology groups} $H_n^{\rm P}(G, \cdot, \lt)$ are defined as the quotients
\[H_n^{\rm P}(G, \cdot, \lt)= \Ker \partial_n /  \Im \partial_{n+1}, \qquad n \in \N.\]
Here $\partial_i$ is the $C_i \to C_{i-1}$ component of $\partial$. 
The choice of the term {\it prismatic}, borrowed from \cite{CIST:HomMCQ}, will be clear after we describe prism-like classifying spaces for this theory.

\section{Simplicial and cubical classifying spaces}\label{S:SimplicesCubes}

\begin{figure}[htb]
\begin{center}
\includegraphics[width=2.1in]{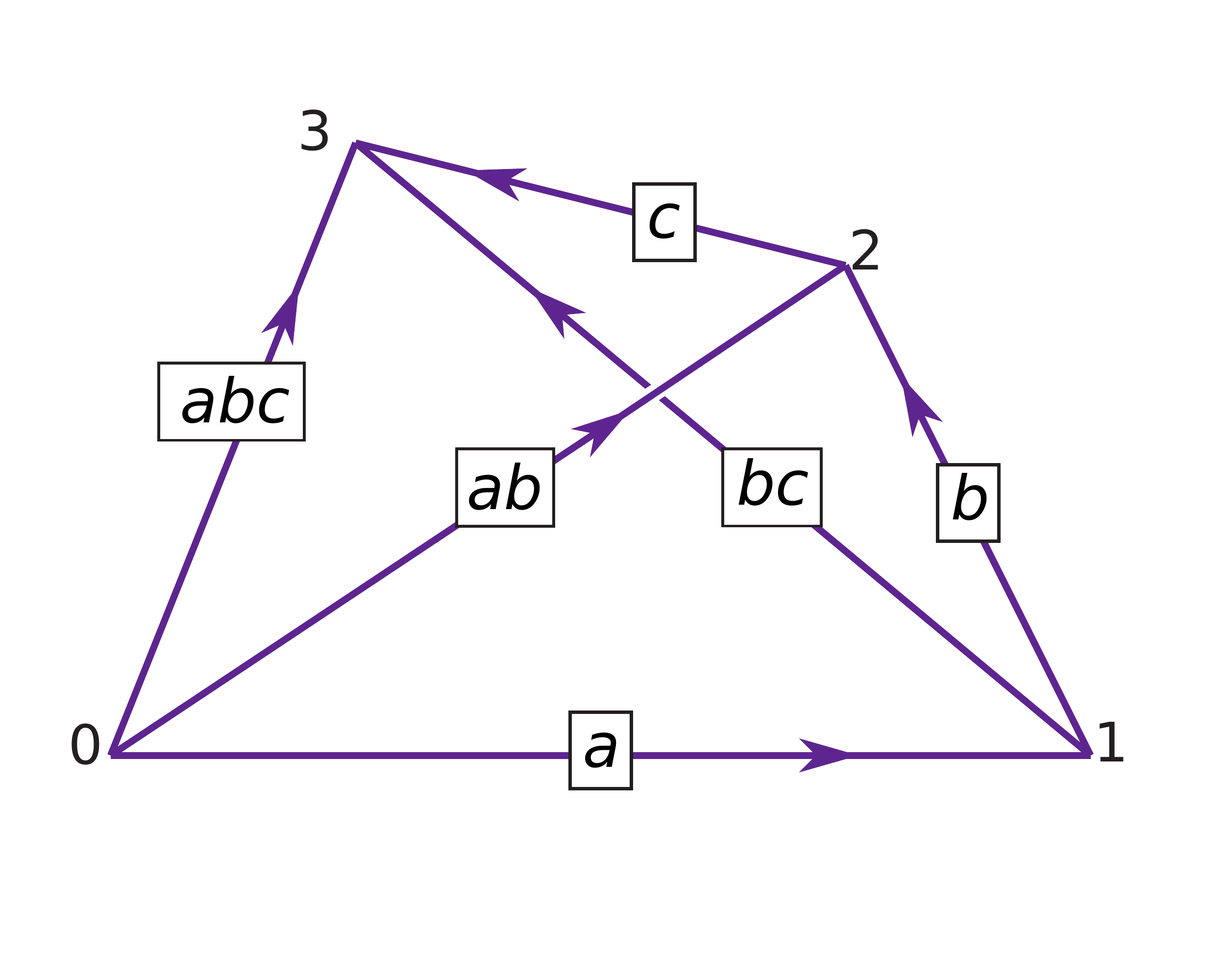}\quad \includegraphics[width=2in]{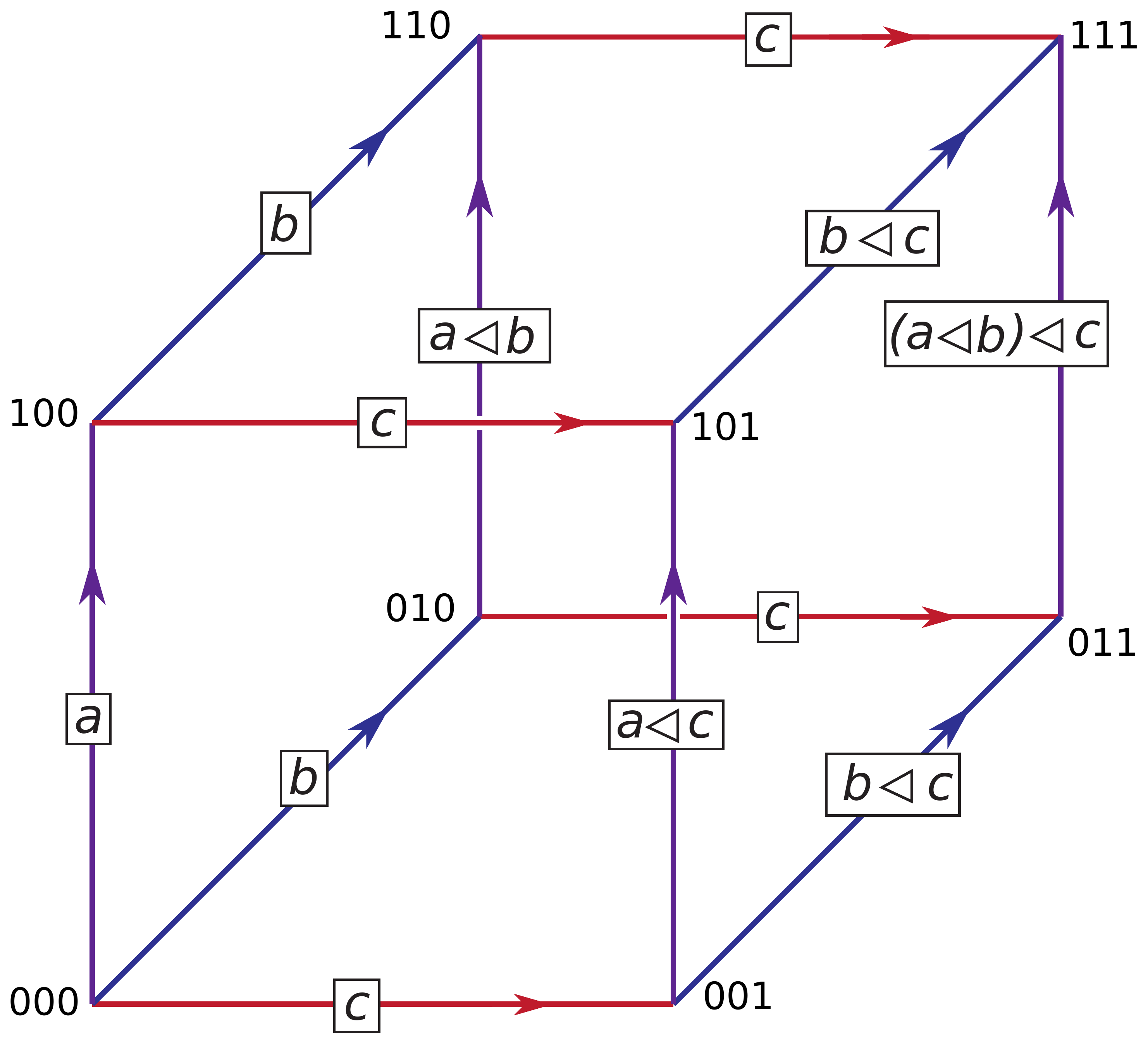}
\end{center}
\caption{The colored tetrahedron with labels $(a,b,c)$ (left), and the colored cube with labels $a|b|c$ (right)}\label{coloredA}
\end{figure}

Recall (from \cite{Cyclic,Emily,LK:SimplKho} for example) that a {\it pre-simplicial set} consists of a collection $(X_n)$, $ n=0,1,2, \ldots$ of sets together with {\it boundary maps} $d^{n}_i \colon X_{n} \to X_{n-1}$ that are defined for $0 \leq i \leq n$, and satisfy
\begin{equation}\label{E:presimpl}
d_i d_j = d_{j-1} d_i \qquad \text{ for } i<j.
\end{equation}
The dependencies on $n$ are typically omitted in the notation.

In case $G$ is a semigroup, the cartesian product $G^n$ can serve as $X_n$, with as boundary maps
\begin{eqnarray*}
d_0(g_1,g_2, \ldots, g_n) &=& (g_2,\ldots, g_n), \\
d_i(g_1,g_2, \ldots, g_n) &=& (g_1,\ldots, g_i\cdot g_{i+1}, \ldots,  g_n), \quad 0< i <n,\\
d_n(g_1,g_2,\ldots, g_n) &=& (g_1, g_2, \ldots, g_{n-1}).
\end{eqnarray*}

Now, a pre-simplicial set can be turned into a chain complex $(\Z X_{n}, \partial_n)$, where $\partial_n$ is the linearization of $\sum_{i=0}^n (-1)^id^{n}_i$. The property $\partial_{n-1} \partial_n=0$ can be obtained either algebraically, using~\eqref{E:presimpl}; or geometrically, using the {\it 
(fat) 
geometric realization} of $X_{n}$. The latter is constructed as follows:
\[|X|=\coprod_{n \geq 0} X_n \times \Delta^n \raisebox{-.15cm}{$\Big/$} \raisebox{-.3cm}{$\sim$}\, .\]
Here
\begin{itemize}
\item the $X_{n}$ are endowed with the discrete topology;
\item $\Delta^n$ is the $n$-simplex with its standard topology,
\[\Delta^n=\{\,(x_0,x_1, \ldots, x_n) \in \mathbb{R}^{n+1} \, :\, \sum_j x_j=1 \ \, \text{ and } \, \ 0\le x_j \,\};\]
\item the equivalence relation $\sim$ is generated by $(x, D^i (p)) \sim (d_i(x), p)$, where the $D_i$ are defined by
\[D^i(x_0, x_1, \ldots, x_{n-1})=(x_0, x_1, \ldots, x_{i-1}, 0, x_{i}, \ldots, x_{n-1}),\]
$0 \leq i \leq n$, and turn the $\Delta^n$ into a pre-cosimplical space (i.e., satisfy relations dual to~\eqref{E:presimpl}).
\end{itemize}
By construction, the homology of the chain complex $(\Z X_{n}, \partial_n)$ coincides with that of the cell complex $|X|$. 
We note that this realization does not take into consideration any degeneracy conditions which is why the adjective ``fat" is sometimes added. 

This geometric realization is particularly enlightening for the pre-simpl\-ic\-ial set $X_n=G^n$ constructed above from a semigroup $G$ (assumed discrete). Note that in this case the space $|X|$ is denoted by $BG$, and called the {\it classifying space} of $G$; the homology obtained is the usual semigroup homology. The space $BG$ consists of simplices whose edges are labeled by $G$ in a particular way. Namely, denote the vertex $(0,\ldots,0,1,0,\ldots,0)$ of $\Delta^n$, where $1$ is at the $i$th position, by $i$. This yields an order on the vertices, which induces an orientation of $\Delta^n$. Given an $n$-tuple $(g_1,\ldots,g_n) \in G^n$, the edges of the corresponding $n$-simplex are labeled as follows: the edge $i \to j$, $i<j$, receives the label $g_{i+1}\cdots g_j$. See Fig.~\ref{coloredA} for an example in dimension~$3$. The boundary map is the usual geometric one. It yields a combination of $G$-labeled $(n-1)$-simplices. In each of them, one should rename the vertices to get the canonical names $0,1,\ldots,n-1$; this is done in the only way that preserves the vertex order.

The whole simplicial story has a cubical counterpart. Namely, recall (for instance from \cite{BH_cubes,LebedVendramin}) that a {\it pre-cubical set} consists of a collection $(X_n)$, $n=0,1,2, \ldots$ of sets together with two families of boundary maps $d^{n; +}_i, d^{n;-}_i \colon X_{n} \to X_{n-1}$, $1 \leq i \leq n$, satisfying
\begin{equation*}%\label{E:precub}
d^{\varepsilon}_i d^{\zeta}_j = d^{\zeta}_{j-1} d^{\varepsilon}_i \qquad \text{ for } i<j \text{ and } \varepsilon, \zeta \in \{+,-\}.
\end{equation*}

In case $G$ is a shelf, the cartesian product $G^n$ can serve as $X_n$, with
\begin{eqnarray*}
d^+_i(g_1,g_2, \ldots, g_n) &=& (g_1\lt g_i,\ldots, g_{i-1}\lt g_i, g_{i+1}, \ldots,  g_n), \quad 1 \leq i \leq n,\\
d^-_i(g_1,g_2, \ldots, g_n) &=& (g_1,\ldots, g_{i-1}, g_{i+1}, \ldots,  g_n), \quad 1 \leq i \leq n.
\end{eqnarray*}

Similarly to pre-simplicial sets, a pre-cubical set can be turned into a chain complex $(\Z X_{n}, \partial_n)$, where $\partial_n$ is the linearization of $\sum_{i=0}^n (-1)^i(d^{n;+}_i-d^{n;-}_i)$. The geometric realization of $X_{n}$ is constructed as follows:
\[|X|=\coprod_{n \geq 0} X_n \times \square^n \raisebox{-.15cm}{$\Big/$} \raisebox{-.3cm}{$\sim$}\, .\]
Here $\square^n=[0,1]^n$ is the standard $n$-cube, and the equivalence relation $\sim$ is generated by $(x, D_{\varepsilon}^i (p)) \sim (d^{\varepsilon}_i(x), p)$, $1 \leq i \leq n$, $\varepsilon \in \{+,-\}$, where
\begin{align*}
D_-^i(x_1, \ldots, x_{n-1})&=(x_1, \ldots, x_{i-1}, 0, x_{i}, \ldots, x_{n-1}),\\
D_+^i(x_1, \ldots, x_{n-1})&=(x_1, \ldots, x_{i-1}, 1, x_{i}, \ldots, x_{n-1}).
\end{align*}
Again, the homology of the chain complex $(\Z X_{n}, \partial_n)$ coincides with that of the cell complex $|X|$.

For the pre-cubical set $X_n=G^n$ constructed above from a shelf $G$, $|X|$ consists of cubes with $G$-labeled edges, and the boundary maps are the usual geometric ones. Given an $n$-tuple $(g_1,\ldots,g_n) \in G^n$, the edges of the corresponding $n$-cube are labeled as follows: the edge
\[(\theta_1, \ldots, \theta_{i-1}, 0, \theta_{i+1}, \ldots, \theta_n) \to (\theta_1, \ldots, \theta_{i-1}, 1, \theta_{i+1}, \ldots, \theta_n),\]
where all $\theta_j \in \{0,1\}$, receives the label $(\cdots (g_i \lt g_{j_1})\lt \cdots )\lt g_{j_p}$, where $j_1 < \cdots < j_p$ are all the indices from $\{i+1,\ldots,n\}$ such that  $\theta_{j}=1$. See Fig.~\ref{coloredA} for the case $n=3$. In this case the homology obtained is the usual homology of shelves, called {\it rack homology}, and the space $BG=|X|$ is called the {\it classifying space}, or the {\it rack space} of $G$ \cite{RackHom}.

If a semigroup $G$ is in fact a monoid (i.e., has a unit $e$), then the maps
\[s_n^j(g_1,g_2, \ldots , g_{n})=(g_1, \ldots, g_j, e, g_{j+1}, \ldots, g_{n}),\]
called {\it degeneracies}, are compatible with the $d_j$ in the sense that the collection $(X_n;d^n_i;s_n^j)$ satisfies the axioms of a {\it simplicial set}. The complex $(\Z X_{n}, \partial_n)$ associated to any  simplicial set has a {\it degenerate sub-complex} $D_n=\sum_j \Z \Im s_n^j$. The quotient, called {\it normalized complex}, generally behaves better than the original complex. The geometric realization of a simplicial set is obtained from that of the underlying pre-simplicial set by adding the relations $(x, S_j (p)) \sim (s^j(x), p)$, where
\[S_j(x_0,\ldots, x_{n+1})=(x_0, \ldots, x_j + x_{j+1}, \ldots, x_{n+1}).\]
If our simplicial set originates from a monoid, as above, then one can handle these additional relations by adding a $2$-cell whose boundary is an $e$-labeled edge. In other words, one is allowed to shrink all $e$-labeled edges. The homology obtained is the classical group homology.

Similarly, suppose that a shelf $G$ is a {\it spindle}---i.e., satisfies axiom $\I$: $a\lt a =a$ for all $a \in G$; any quandle will do. Then the degeneracies
\[\widetilde{s}_n^{\; j}(g_1,\ldots, g_{n})=(g_1, \ldots, g_{j-1}, g_{j}, g_{j}, g_{j+1}, \ldots, g_{n})\]
and the boundary maps $d^{n;\pm}_i$ above turn the collection $(G^n)$ into a {\it weak skew cubical} set. See \cite{LebedVendramin} for the definition, and its difference from the more classical notion of cubical sets. Again, the quotient of $(\Z X_{n}, \partial_n)$ by the {\it degenerate sub-complex} $D_n=\sum_j \Z \Im \widetilde{s}_n^{\; j}$ is called {\it normalized complex}, and behaves better than the original complex. It is used, for instance, for constructing quandle cocycle invariants for knots; see Section~\ref{S:Invariants} for more details. In the geometric realization of a weak skew cubical set, one needs to add the relations $(x, \widetilde{S}_j (p)) \sim (\widetilde{s}^{\; j}(x), p)$, where
\[\widetilde{S}_j(x_0,\ldots, x_{n+1})=(x_0, \ldots, x_j + x_{j+1}-x_jx_{j+1}, \ldots, x_{n+1}).\]
This map can be thought of as a squeezing onto the diagonal hyperplane $x_j = x_{j+1}$, followed by a rescaling\footnote{This explains the term \emph{skew}.}. If our weak skew cubical set originates from a spindle, then one can handle these extra relations by adding extra cells. For instance, for any $a \in G$, one needs one extra $3$-cell whose boundary is a square with four $a$-labeled edges (Fig.~\ref{F:degeneracies}). 
\begin{figure}[htb]
\begin{center}
\includegraphics[width=0.25\paperwidth]{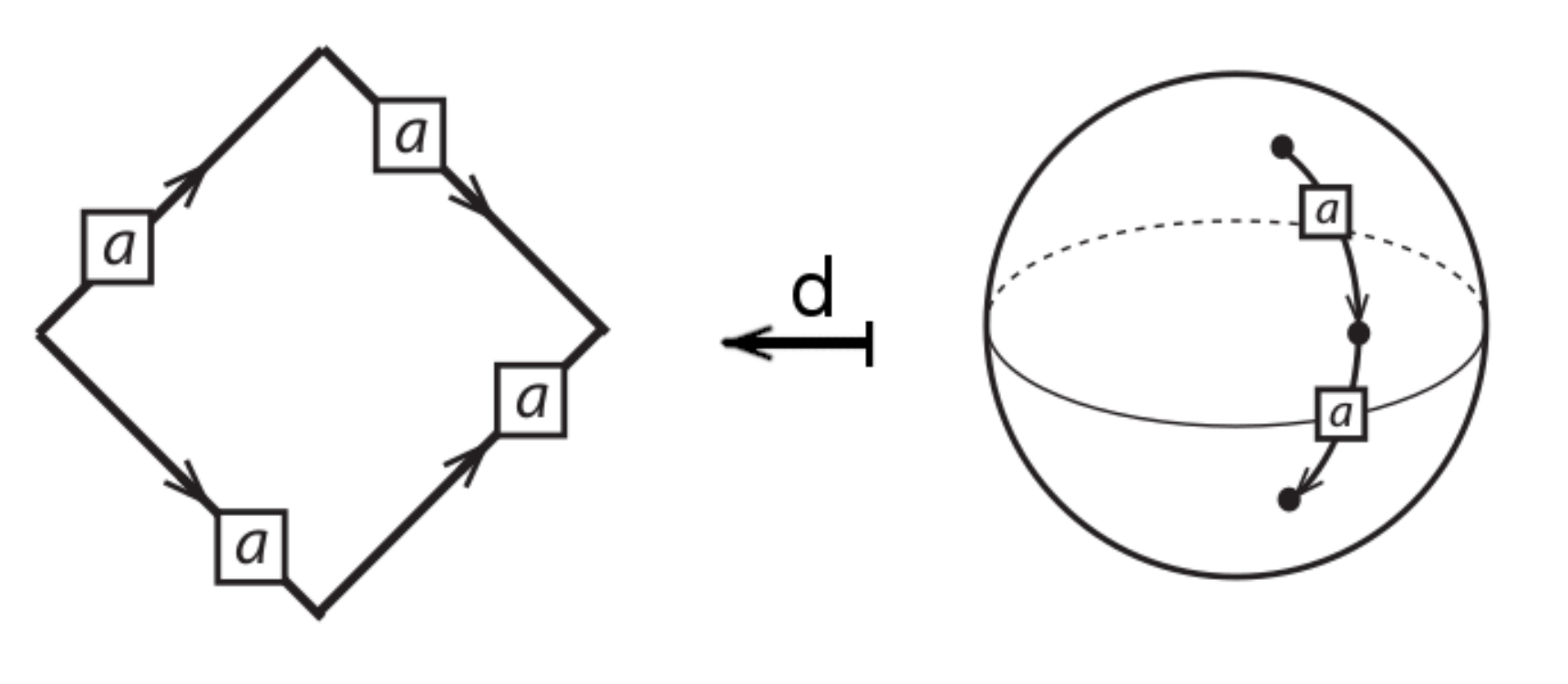}
\end{center}
\caption{Degenerate $2$-cells and additional $3$-cells for a quandle space}\label{F:degeneracies}
\end{figure}
In this case, the {rack space} is transformed into the {\it quandle space} from \cite{Nosaka_QuSpace, Nosaka_QuSpace2, Yang_ExQuSpace}, and we get the quandle homology from \cite{QuandleHom}.

The connection between the homology of associative structures and simplicial sets is in fact very strong: as shown in \cite{SimplIsHoch}, the cohomology of any locally finite simplicial complex is the Hochschild cohomology of certain explicitly constructed unital associative algebra. It would be interesting to find a similar counterpart of the weak skew cubical cohomology. A reasonable candidate seems to be the cohomology of set-theoretic solutions to the Yang--Baxter equation, of which the rack cohomology is a particular case \cite{HomologyYB, LebedVendramin}.

\section{Prismatic classifying spaces}\label{S:all}

We now turn to the central construction of this paper---the classifying space $BG$ of a shalgebra $(G,\cdot,\lt)$\footnote{We use the same notation $BG$ for all classifying spaces in this paper. Its precise meaning is determined by the object $G$  we are working with.}. It is inspired by the classifying spaces of a semigroup and a shelf from the previous section, and is a merge of the two. We will start with a detailed description of the lower dimensional skeleton of $BG$, which is of particular importance for the purposes of this paper.
 We will then explain what happens in higher dimensions, and show that the prismatic homology $H^{\rm P}(G, \cdot, \lt)$ of our shalgebra can be understood as the homology of the prismatic space $BG$.

\subsection{Generating prisms in dimensions up to $3$}\label{SS:123}

The cell complex $BG$ has a unique $0$-dimensional cell\footnote{Prismatic homology has a more general version, where coefficients in a shalgebra module $M$ are added to the construction. In this case we need one $0$-dimensional cell per element of~$M$. The generalization of our prismatic classifying space to this setting is straightforward, but we omit it in order not to overload the presentation.}. Its $1$-dimensional cells are indexed by the elements of $G$. They have the form $(a,\Delta^1)$, $a \in G$, and can be thought of as $G$-labeled edges, as indicated in Fig.~\ref{geomreal1} (left).

\begin{figure}[htb]
\begin{center}
\includegraphics[width=3in]{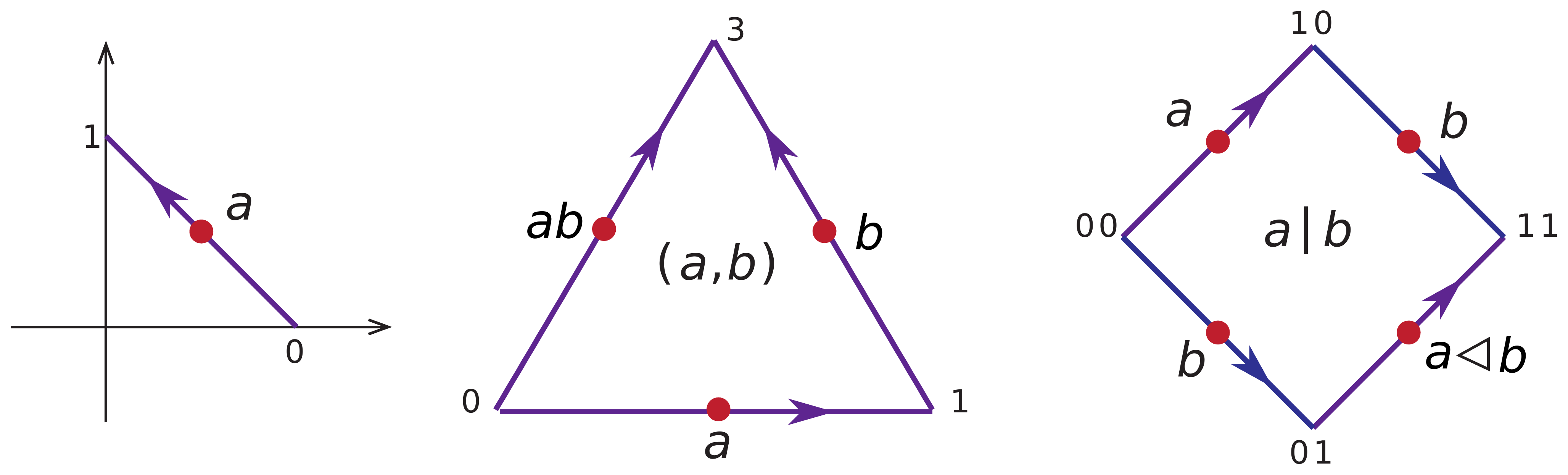}
\end{center}
\caption{Labeled edges, triangles, and squares}\label{geomreal1}
\end{figure}

The same Figure depicts
$2$-dimensional cells, which are either triangles or squares. For each pair of elements in $G$ there is a triangle that is labeled $(a,b)$ and a square that is labeled $a|b$. The boundary of such a $2$-cell is labeled using $a$, $b$, and either $ab$ or $a\lt b$ in the triangle and square case respectively.

\begin{figure}[htb]
\begin{center}
\includegraphics[width=2.5in]{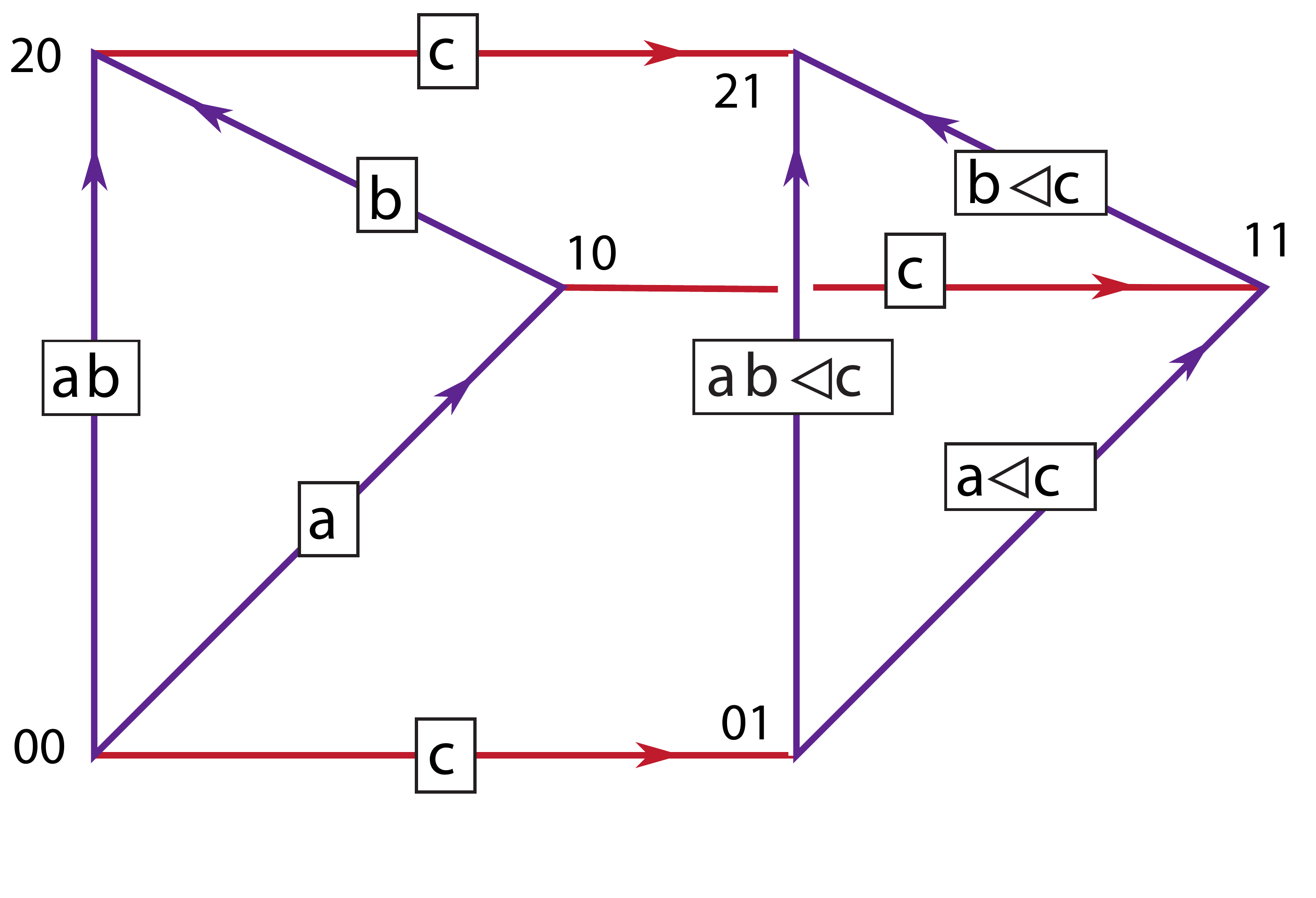} \includegraphics[width=2in]{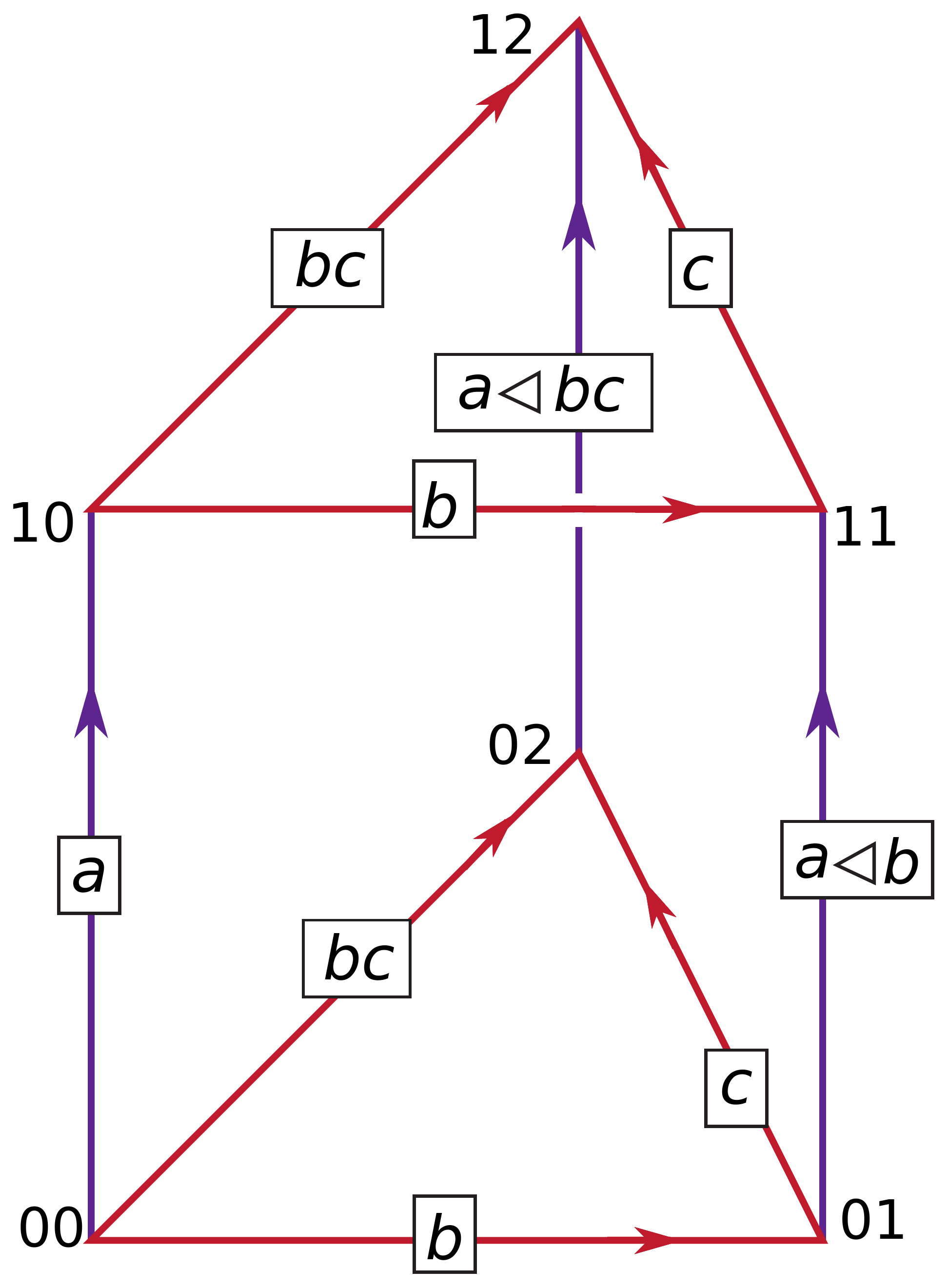}
\end{center}
\caption{The colored prisms with labels $(a,b)|c$ (left) and $a|(b,c)$ (right)}\label{coloredYI}
\end{figure}

In dimension $3$, there are four different types of prismatic $3$-cells: the two intermediate types are indeed triangular prisms $\Delta^2\times\Delta^1$ and $\Delta^1\times\Delta^2$ while the extreme cases are the tetrahedron $\Delta^3$ and the cube $\Delta^1 \times \Delta^1 \times \Delta^1.$ These are labeled by triples of elements in the shalgebra $G$. Thus for each triple $(a,b,c)\in G^3$ there is a tetrahedron in $BG$; for each triple written as  $(a,b)|c$ there is a prism $\Delta^2\times \Delta^1$; for each triple $a|(b,c)$ there is a prism $\Delta^1\times \Delta^2$; and there are cubes labeled by triples $a|b|c$. In Figures~\ref{coloredA} and \ref{coloredYI},
 these $3$-dimensional polyhedra are indicated with labels upon their edges. The $2$-dimensional faces are squares and triangles. We trust that the reader can supply the labels to these faces
and recover the boundary formulas for $\partial(a,b,c)$, $\partial((a,b)|c)$, $\partial (a|(b,c))$, and $\partial(a|b|c)$ from Section~\ref{S:PrismHom}.

\subsection{Generating prisms in dimension $4$}\label{SS:4}

%1
\begin{figure}
\begin{center}
\includegraphics[width=5.5cm]{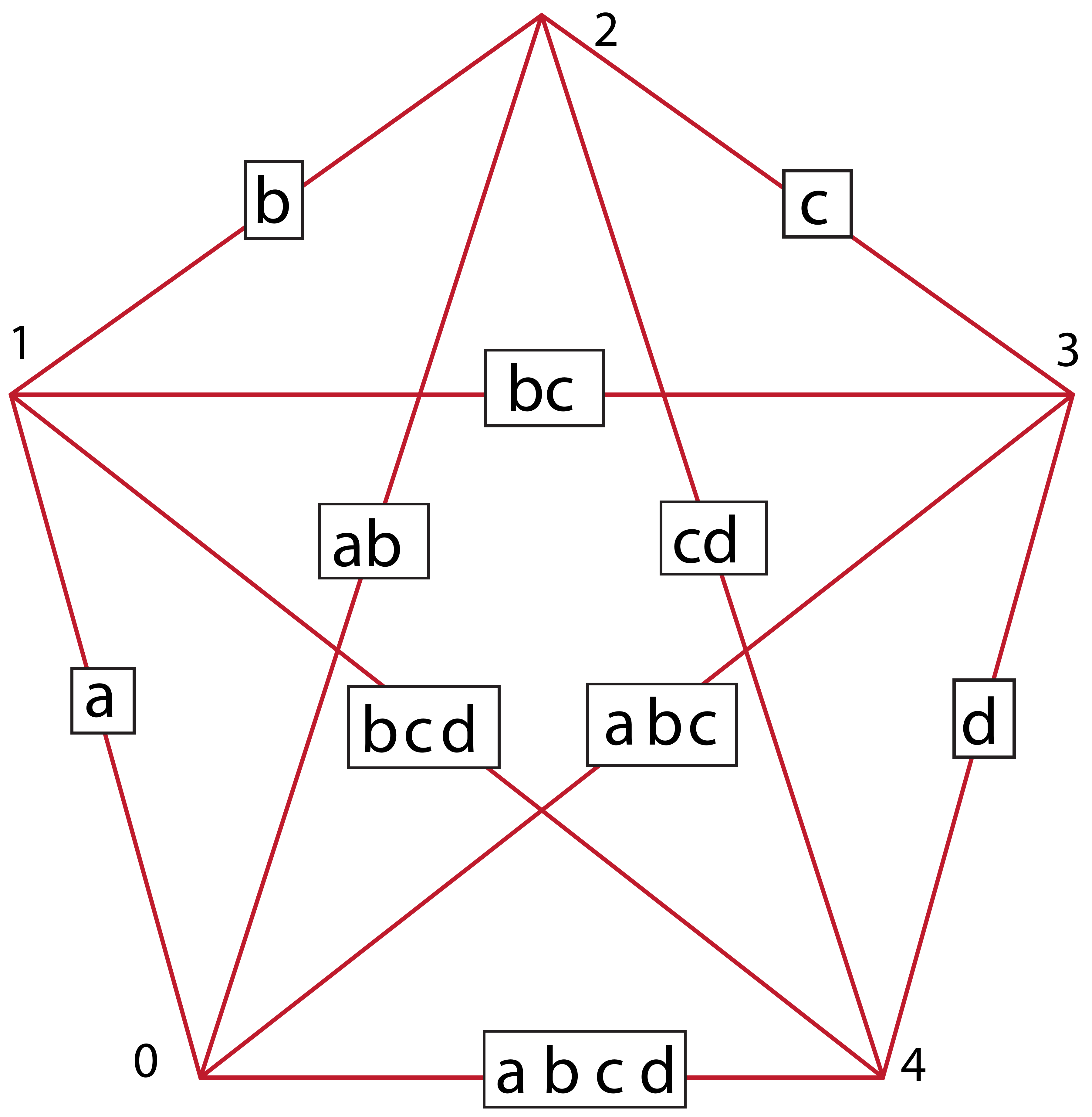}
\end{center}
\caption{The colored $\Delta^4$ with labels $(a,b,c,d)$}\label{coloredpentagram}
\end{figure}

The authors anticipate that most readers will be familiar, at best, with the $4$-dimensional polytopes that are the $4$-simplex $\Delta^4$ and the $4$-dimensional hypercube $\Delta^1 \times \Delta^1 \times\Delta^1 \times\Delta^1$. These two extreme $4$-dimensional cells are labeled by quadruples of shalgebra elements $(a,b,c,d)$ and $a|b|c|d$, respectively.

%8
\begin{figure}
\begin{center}
\includegraphics[width=8cm]{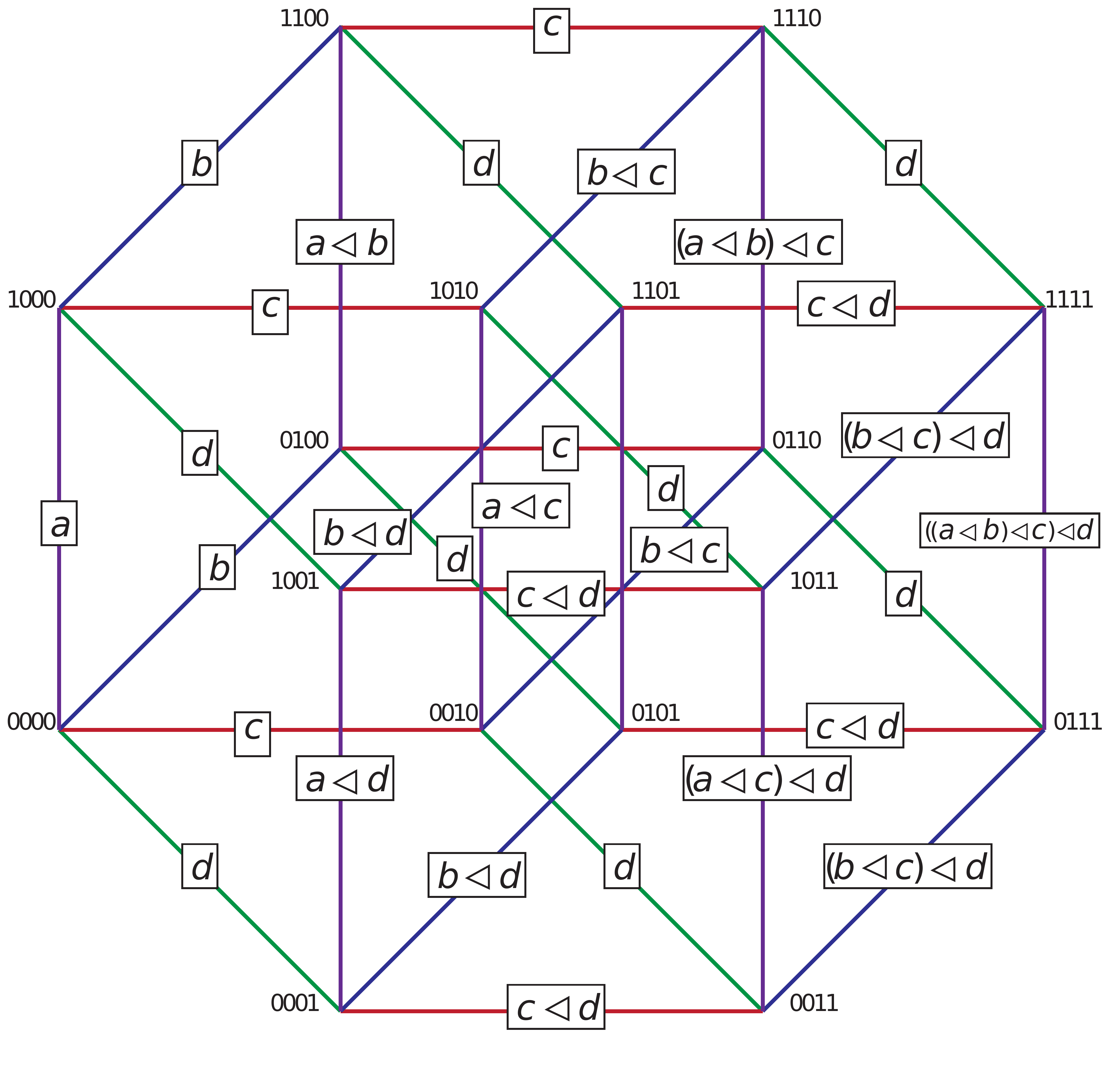}
\end{center}
\caption{The colored $\Delta^1 \times \Delta^1 \times \Delta^1 \times \Delta^1$ with labels $a|b|c|d$}\label{coloredhypercube}
\end{figure}

%2
\begin{figure}
\begin{center}
\includegraphics[width=7.5cm]{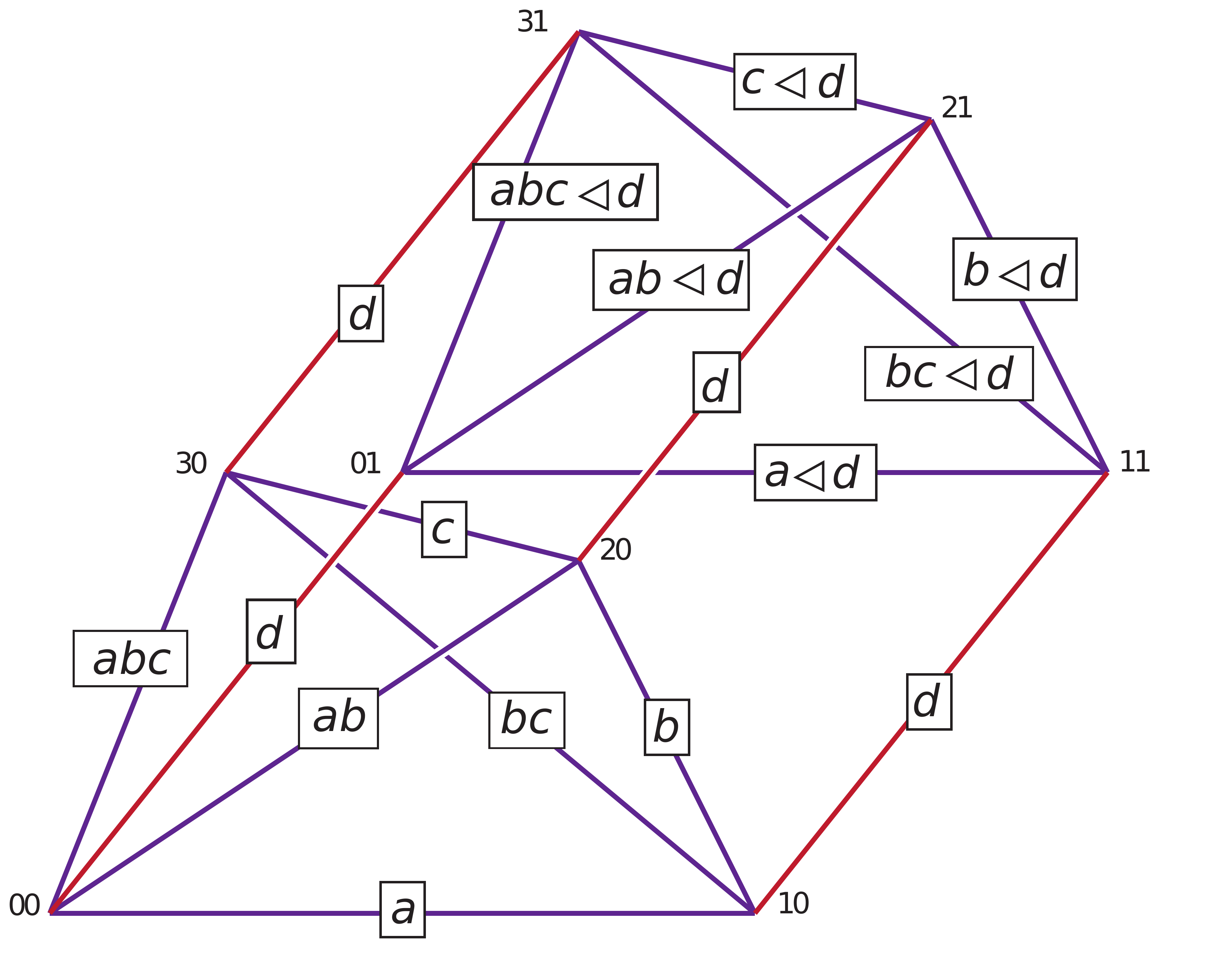}
\end{center}
\caption{The colored $\Delta^3 \times \Delta^1$ with labels $(a,b,c)|d$}\label{coloredAI}
\end{figure}

In the case of the $4$-simplex $\Delta^4$, its $1$-skeleton is the complete graph on $5$ vertices.
 The edges are labeled by products of four shalgebra elements $a,b,c,d$, as explained in Section~\ref{S:SimplicesCubes}. This labeling is depicted in Fig.~\ref{coloredpentagram}; here and below edge orientation is omitted for readability. The $3$-dimensional faces are tetrahedra with inherited labels; they yield the boundary $\partial(a,b,c,d)$, as computed in Section~\ref{S:PrismHom}.

%3
\begin{figure}
\begin{center}
\includegraphics[width=7cm]{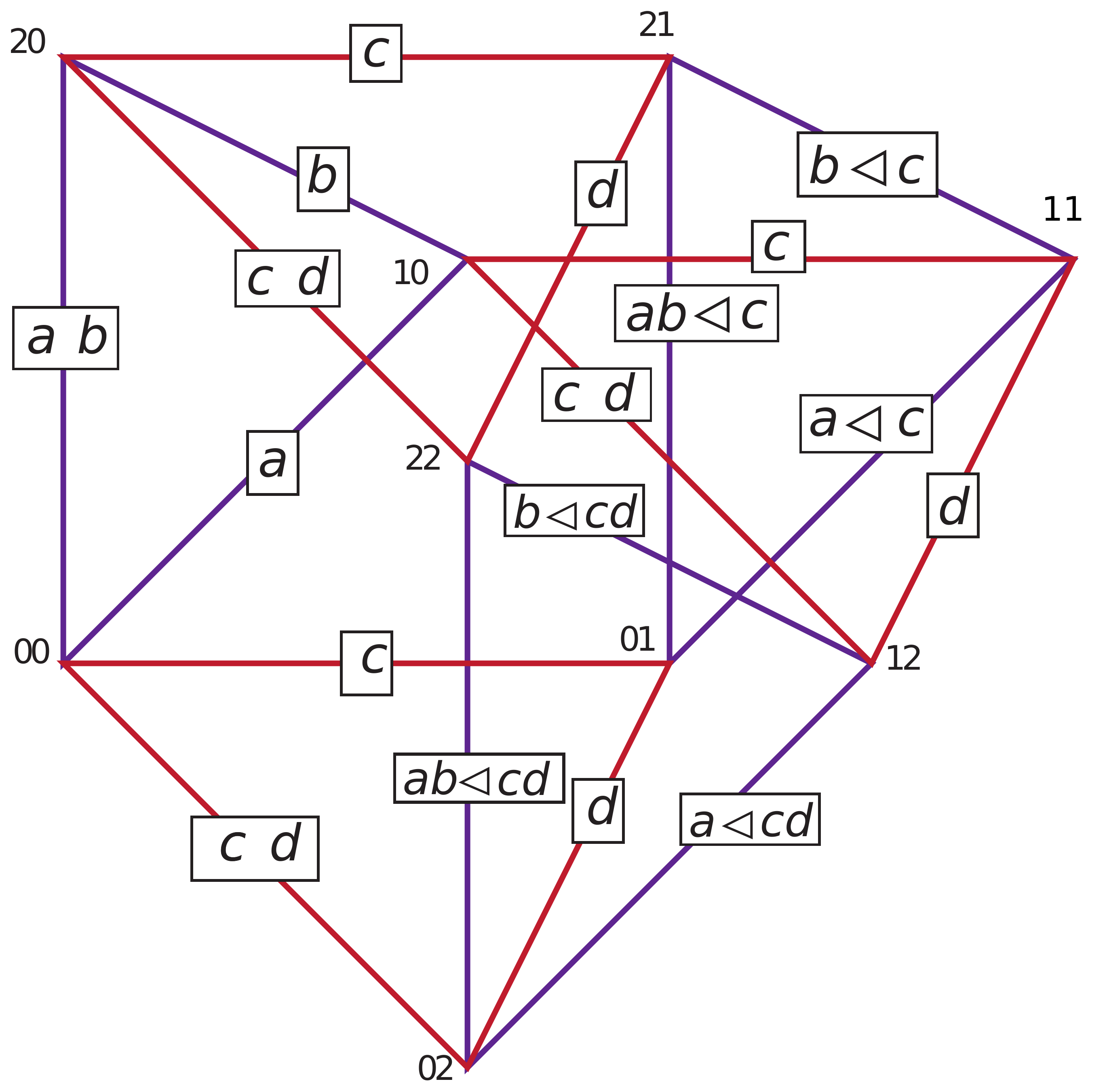}
\end{center}
\caption{The colored $\Delta^2 \times \Delta^2$ with labels $(a,b)|(c,d)$}\label{coloredYY}
\end{figure}

%4
\begin{figure}
\begin{center}
\includegraphics[width=7.5cm]{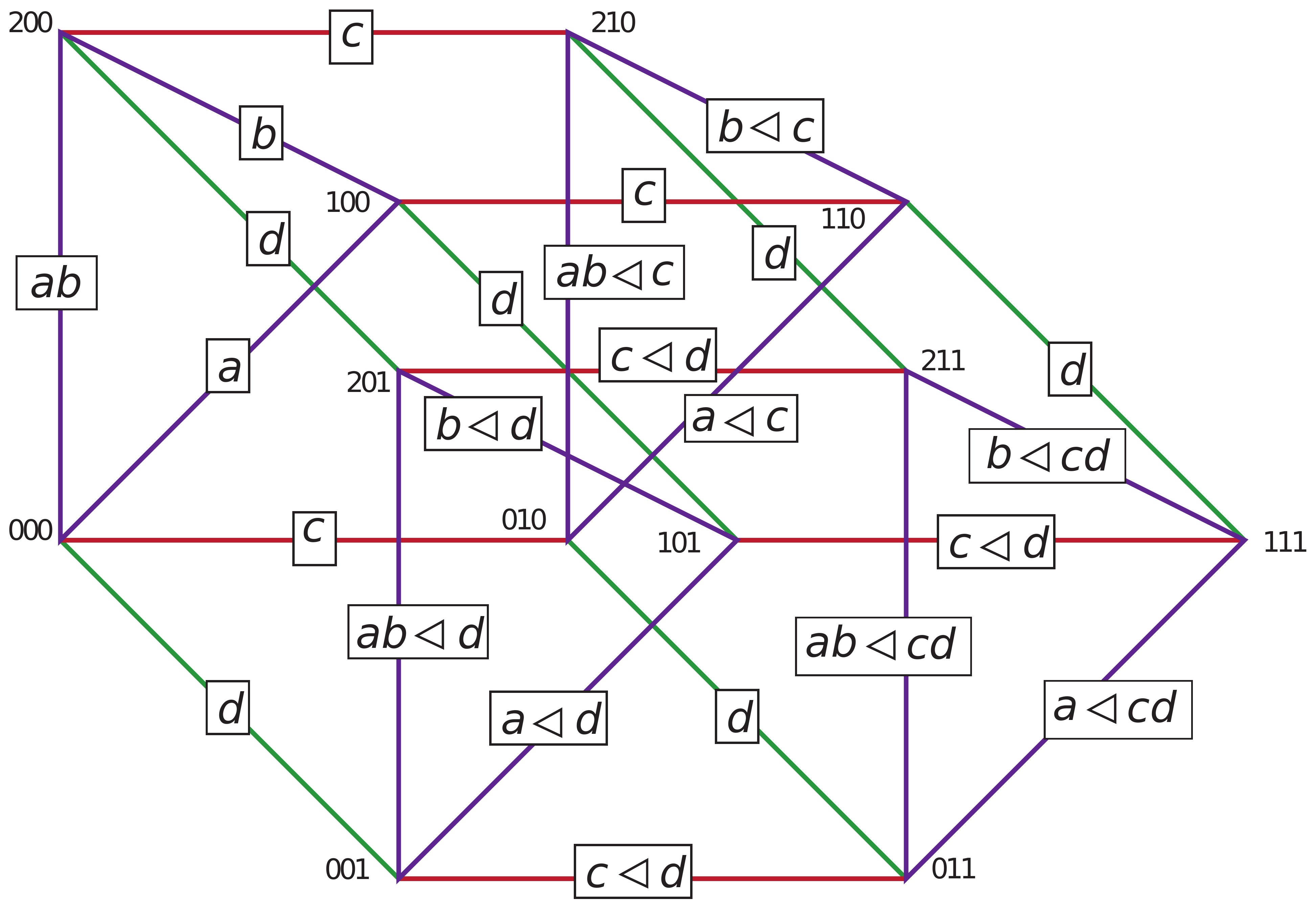}
\end{center}
\caption{The colored $\Delta^2 \times \Delta^1 \times \Delta^1$ with labels $(a,b)|c|d$}\label{coloredYII}
\end{figure}

%5
\begin{figure}
\begin{center}
\includegraphics[width=5cm]{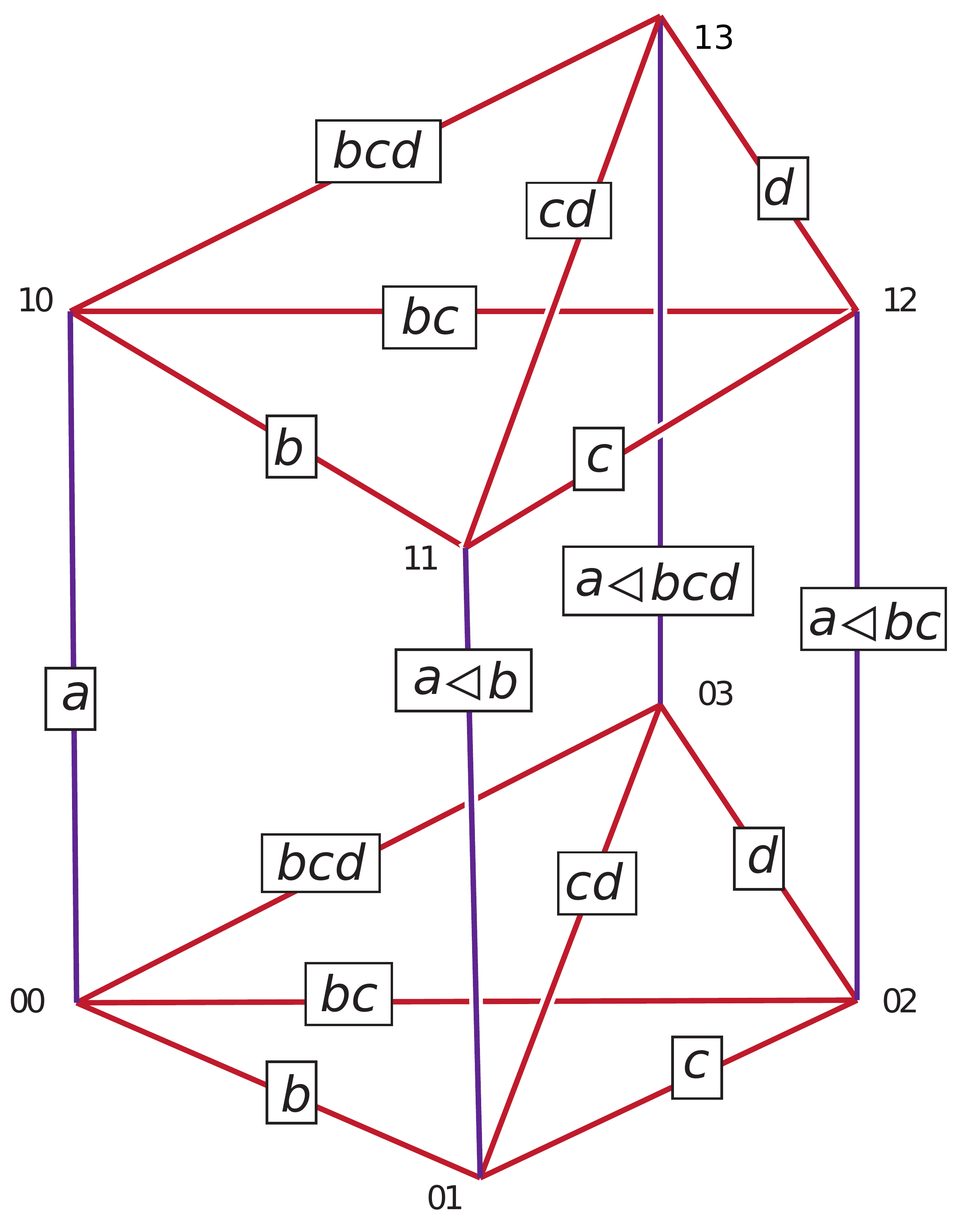}\quad \includegraphics[width=6.cm]{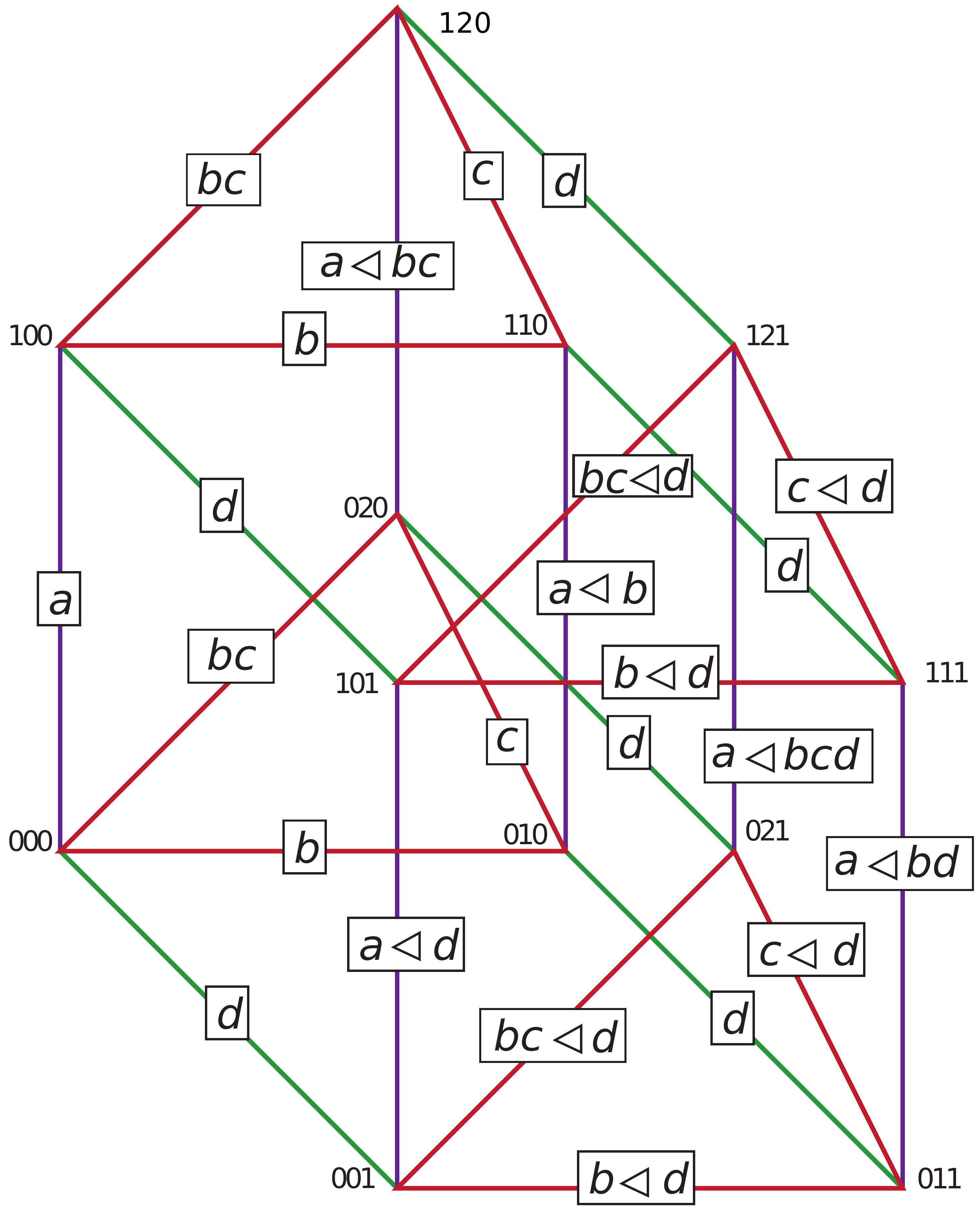}
\end{center}
\caption{The colored $\Delta^1 \times \Delta^3$ with labels $a|(b,c,d)$ (left) and the colored $\Delta^1 \times \Delta^2 \times \Delta^1$ with labels $a|(b,c)|d$ (right)}\label{coloredIA}
\end{figure}

%6 \label{coloredIYI}

%7
\begin{figure}
\begin{center}
\includegraphics[width=6.5cm]{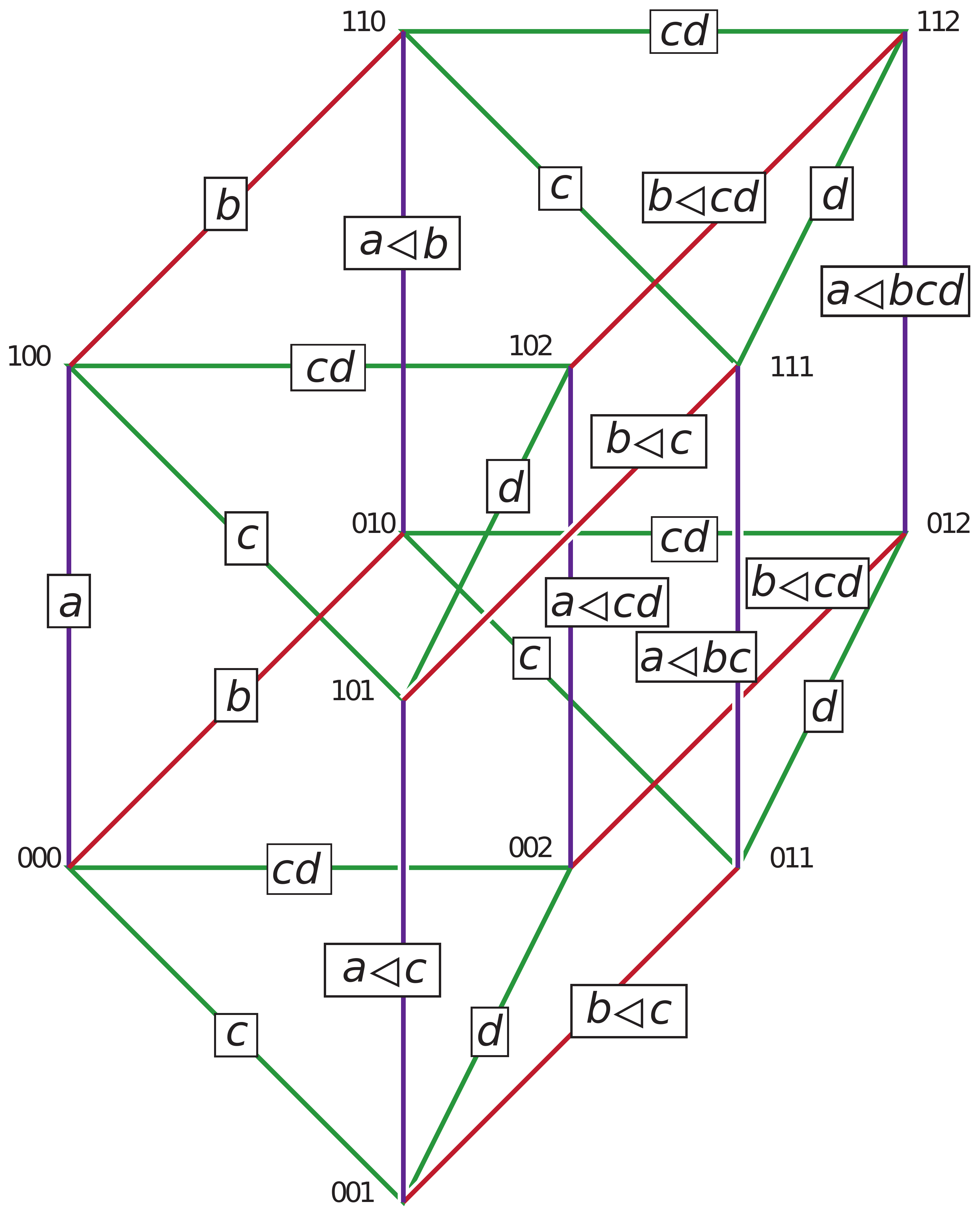}
\end{center}
\caption{The colored $\Delta^1 \times \Delta^1 \times \Delta^2$ with labels $a|b|(c,d)$}\label{coloredIIY}
\end{figure}

The $1$-skeleton of the colored hypercube is indicated in Fig.~\ref{coloredhypercube}.
It is labeled by quadruples of the form $a|b|c|d$. Its boundary consists of eight $3$-dimensional cubes whose induced labels are those predicted by the boundary formula $\partial (a|b|c|d)$ in Section~\ref{S:PrismHom}. The labels that are indicated upon the edges should aid the reader to supply labels for the $2$-dimensional square faces and $3$-dimensional cubical solids.

The remaining six prismatic sets have their $1$-skeletons illustrated in Figures~\ref{coloredAI} through \ref{coloredIIY}. We encourage the diligent reader to discover the various $3$-dimensional faces of these and match them with the boundaries of the corresponding chains $(a,b,c)|d$, $(a,b)|(c,d)$, $(a,b)|c|d$, $a|(b,c,d)$, $a|(b,c)|d$, and $a|b|(c,d)$.

\subsection{Higher dimensional prisms}\label{SS:n}

By this point the reader might have developed an intuition of what the cells of dimension $n$ should look like: they are products of simplices with $G$-labeled edges, all the labels being determined by an $n$-tuple of elements of $G$. More rigorously, the \emph{classifying space of a shalgebra $(G,\cdot,\lt)$} is defined as follows:
\[BG=\coprod_{n \geq 0}\,\coprod_{\overline{k} \in P_n}\, G^{n} \times \Delta^{\overline{k}} \raisebox{-.15cm}{$\Big/$} \raisebox{-.3cm}{$\sim$}\, .\]
Here
\begin{itemize}
\item $P_n$ is the set of all ordered partitions ${\overline{k}}=(k_1,\ldots,\k_\ell)$ of $n$, i.e., $k_1+\cdots +\k_\ell =n$, $\ell \in \N$, $k_i \geq 1$;
\item $\Delta^{(k_1,\ldots,\k_\ell)} := \Delta^{k_1} \times \Delta^{k_2} \times \cdots \times \Delta^{k_\ell}$.
\end{itemize}
In the remainder of this section we will define the identification $\sim$, which glues all these generalized prisms together, and show that the homology of $BG$ computes the prismatic homology of $G$.

Given a bracketed $n$-tuple
\[\overline{g}:=(g_1, \ldots, g_{k_1})|\ldots |(g_{s +1}, \ldots, g_{s+k_j})|\ldots |(g_{t+1}, \ldots, g_n) \in G^n,\]
where $s=\sum_{i=1}^{{j-1}}k_i$ and $t=\sum_{i=1}^{{\ell-1}}k_i$, we interpret $(\overline{g}, \Delta^{\overline{k}})$ as the generalized prism $\Delta^{\overline{k}}$ with $G$-labeled edges. Concretely, with the notations from the description of the simplices $\Delta^m$ given in Section~\ref{S:SimplicesCubes}, the edge
\[\{p_1\} \times \cdots \times \{p_{q-1}\} \times (p_q \to p'_q) \times \{p_{q+1}\} \times \cdots \times \{p_{\ell}\}, \qquad p_q < p'_q\]
receives the label
\begin{equation}\label{E:LabelingRule}
(g_{p_q+1} \cdots g_{p'_q}) \lt (g_{q+1;p_{q+1}}\cdots g_{\ell;p_{\ell}}),
\end{equation}
where $g_{u;p_{u}} = g_{s+1} \cdots g_{s+p_{u}}$, and $s=\sum_{t=1}^{{u-1}}k_t$. A $G$-labeling constructed in this manner is called \emph{good}. The bracketed $n$-tuple $\overline{g}$ is called the \emph{label} of the generalized prism $\Delta^{\overline{k}}$ endowed with a good $G$-labeling as above.

Good labelings are best understood inductively. Indeed, let us explain how to label $((\overline{g},\overline{h}), \Delta^{\overline{k}} \times \Delta^m)$ if you know what to do with $(\overline{g}, \Delta^{\overline{k}})$. Here $\overline{k} \in P_n$, $\overline{g} \in G^n$, $\overline{h} \in G^m$. Put a copy of $\Delta^{\overline{k}}$ at each vertex of the simplex $\Delta^m$, which you label by $\overline{h}$ as described in Section~\ref{S:SimplicesCubes}. Further, replace each edge $i \to j$ of $\Delta^m$ by $n+1$ parallel copies with the same $G$-label; the $p$th copy ($p=0,1,\ldots, n$) connects the $p$th vertices of the copies of $\Delta^{\overline{k}}$ placed at the vertices $i$ and $j$ of $\Delta^m$. Finally, label the copy of $\Delta^{\overline{k}}$ placed at the vertex $i$ by $\overline{g} \lt (h_1 \cdots h_i)$, which you know how to do by assumption; recall that the $\lt$-action of $G$ on itself extends to $G^n$ diagonally. Fig.~\ref{F:LabelsInductively} illustrates this inductive step; a double arrow represents here six parallel copies of the same $G$-labeled arrow. It is easy to verify (for instance by induction) that the explicit and the inductive labeling rules yield the same result.

\begin{figure}
\begin{center}
\includegraphics[width=6cm]{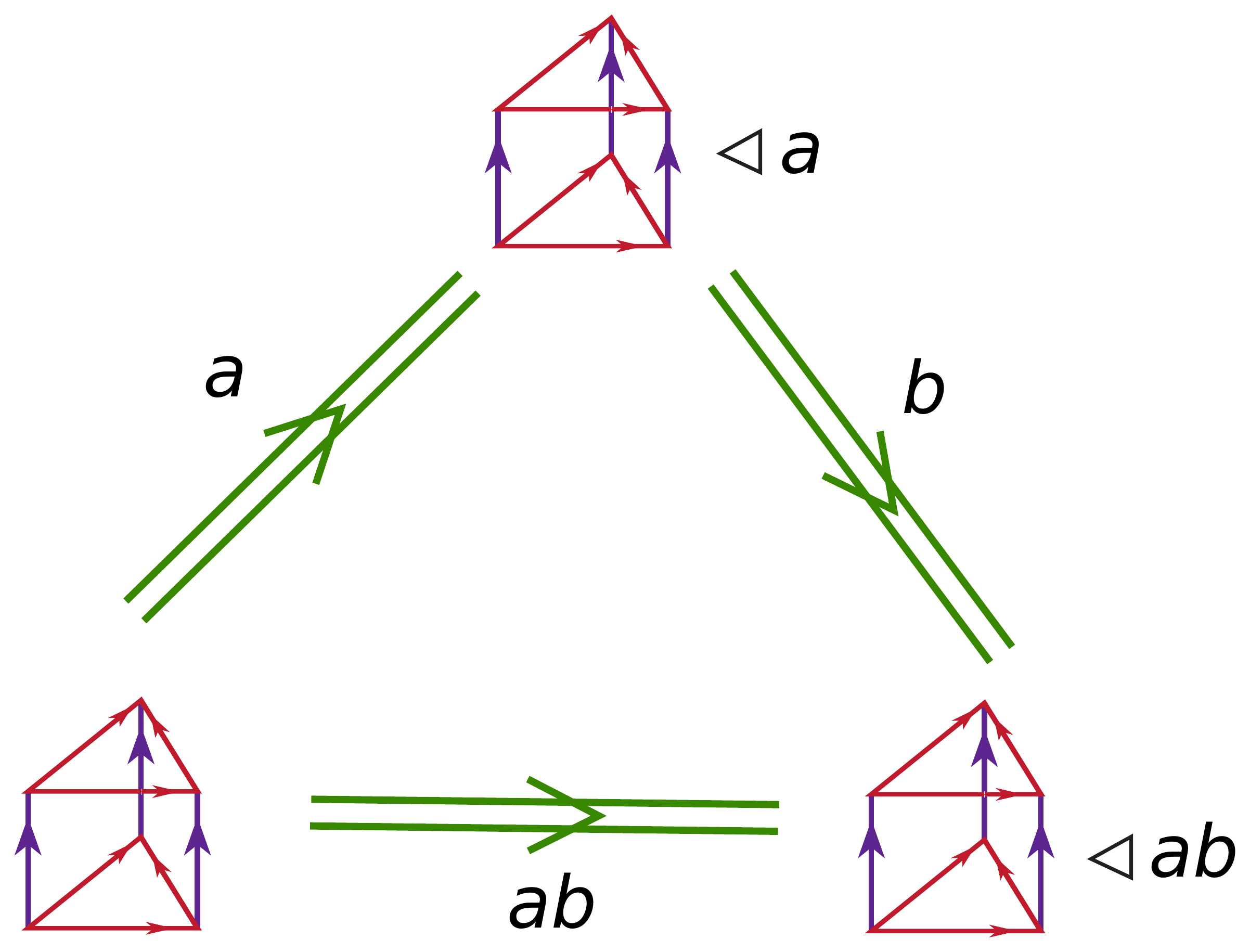}
\end{center}
\caption{Inductive labeling of generalized prisms}\label{F:LabelsInductively}
\end{figure}

\begin{lemma}\label{L:LabelsInductively}
If an edge of the labeled prism $(\overline{g}, \Delta^{\overline{k}})$ receives the label $a \in G$, then, for any $b \in G$, the same edge receives the label $a \lt b$ in $(\overline{g} \lt b, \Delta^{\overline{k}})$.
\end{lemma}

\begin{proof}
The labeling rule \eqref{E:LabelingRule} involves only shalgebra operations, and, by Axioms $\Y\I$ and $\I\I\I$, the $\lt$-action intertwines with both operations.
\end{proof}

\begin{lemma}\label{L:LabelsInductively2}
In the inductive labeling procedure above, the labels of the corresponding edges in the copies of $\Delta^{\overline{k}}$ placed at the vertices $i$ and $j$ of $\Delta^m$, with $i < j$, differ by the $\lt$-action of the label of the edge $i \to j$ in $(\overline{h},\Delta^m)$.
\end{lemma}

\begin{proof}
Let $a$ be the color of this edge in $(\overline{g}, \Delta^{\overline{k}})$. By Lemma~\ref{L:LabelsInductively}, in the copy of $\Delta^{\overline{k}}$ placed at the vertices $i$ and $j$ the same edge receives the labels $a_i:=a \lt (h_1 \cdots h_i)$ and $a_j:=a \lt (h_1 \cdots h_j)$ respectively. By our labeling rules for simplices, the edge $i \to j$ is labeled by $h_{i,j}:=h_{i+1} \cdots h_j$ in $(\overline{h},\Delta^m)$. But then, by Axioms $\A$ and $\I\Y$,
\begin{align*}
a_j&=a \lt (h_1 \cdots h_j)=a \lt ((h_1 \cdots h_i)(h_{i+1} \cdots h_j))\\
&= (a \lt (h_1 \cdots h_i)) \lt (h_{i+1} \cdots h_j)= a_i \lt h_{i,j},
\end{align*}
as desired.
\end{proof}

With these two lemmas in mind, one can envision that in $(\overline{g}, \Delta^{\overline{k}})$ each $g_i$ determines the labels in the $i$th cardinal direction. The dimension $4$ examples above should make this idea clear.

Now, a good labeling of $\Delta^{\overline{k}}$ induces a $G$-labeling of the edges of all its $(n-1)$-dimensional faces, which are also generalized prisms. Below we will prove that this induced labeling is good. The relation $\sim$ is then defined as the identification of these faces with the corresponding cells in $G^{n-1} \times \Delta^{\overline{k'}}$ for suitable $\overline{k'}$.

To manipulate induced labelings of faces, we need to introduce notations for all components of the restriction $\partial_{\overline{k}}$ of the prismatic differential $\partial_n$ from Section~\ref{S:PrismHom} to the summand $\Z[G^{\overline{k}}]$ of~$C_n$. Here
\[G^{\overline{k}} = G^{k_1}\times G^{k_2} \times \cdots \times G^{k_\ell}.\]
First, rewrite Equation~\eqref{E:PrismHom} defining the part $\partial_{j;k_j}$ of the prismatic differential affecting the $j$th factor $G^{k_j}$ of $G^{\overline{k}}$ as
\[\partial_{j;k_j} = \sum_{i=0}^{k_j} (-1)^i d^{k_j}_{j;i}.\]
The component $d^{k_j}_{j;i}$ bears on the entries $i$ and $i+1$ of $G^{k_j}$, and anything to the left of $G^{k_j}$ if $i=0$. 
This component sends $G^{\overline{k}}$ to $G^{\overline{k}-1_j}$, where
\[\overline{k}-1_j:=\begin{cases} k_1+\cdots+(k_j-1)+\cdots +\k_\ell, & k_j>1\\
k_1+\cdots+k_{j-1}+k_{j+1}+\cdots +\k_\ell, & k_j=1.\end{cases}\]
As usual, we will often omit the sub- or superscript $k_j$. Next, the Leibniz rule translates as
\[\partial_{\overline{k}} = \sum_{j=1}^{\ell} (-1)^{k_1+\cdots+k_{j-1}} \partial_{j;k_j} = \sum_{j=1}^{\ell} \sum_{i=0}^{k_j} (-1)^{k_1+\cdots+k_{j-1}+i} d_{j;i}.\]

\begin{theorem}
Take a partition $n=k_1+\cdots +\k_\ell$ and a correspondingly partitioned $n$-tuple $\overline{g} \in G^{\overline{k}}$. Then the boundary of the $G$-labeled generalized prism $(\overline{g}, \Delta^{\overline{k}})$ as defined above consists of $(n-1)$-dimensional cells labeled by $(d_{j;i}(\overline{g}),\Delta^{\overline{k}-1_j})$ with orientation determined by the sign $(-1)^{k_1+\cdots+k_{j-1}+i}$; here $j=1,2,\ldots,\ell$, and $i=0,1,\ldots,k_j$.
\end{theorem}

The theorem directly implies that the relation $\sim$ is well defined, and that the space $BG$ has the same homology as the shalgebra $G$:
\[H_n^{\rm P}(G, \cdot, \lt) \cong H(BG).\]

\begin{proof}
Ignoring the $G$-labels, one can easily compute the geometric boundary of the product of simplices $\Delta^{\overline{k}}$:
\begin{align*}
\delta(\Delta^{\overline{k}}) &=\bigcup_{j=1}^{\ell} (-1)^{k_1+\cdots+k_{j-1}} \Delta^{k_1} \times \cdots \times \delta(\Delta^{k_j})\times \cdots \times \Delta^{k_{\ell}},\\
\delta(\Delta^{k_j}) &= \bigcup_{i=0}^{k_j} (-1)^{i}\Delta^{k_j;\hat{i}},
\end{align*}
where $\Delta^{k_j;\hat{i}}$ is the face of $\Delta^{k_j}$ which does not contain the vertex $i$.

It remains to check that the $G$-labeling on the edges of
\[\delta_{j;i}(\Delta^{\overline{k}}):=\Delta^{k_1} \times \cdots \times \Delta^{k_j;\hat{i}}\times \cdots \times \Delta^{k_{\ell}}\]
induced by $\overline{g}$ coincides with the good labeling constructed from $\overline{g'}:=d_{j;i}(\overline{g})$. We will treat only the case $i=0$ here; the remaining cases $0 < i < k_j$ and $i=k_j$ are straightforward.

In the case $i=0$, for edges from the face $\delta_{j;i}(\Delta^{\overline{k}})$, the term $g_{j;p_{j}}$ appears in the labeling rule~\eqref{E:LabelingRule} only with $p_{j} > 0$. So, $g_{j;p_{j}} = a g'_{j;p_{j}-1}$, where $a:=g_{k_1+\cdots+k_{j-1}+1}$. Using the shalgebra axioms, one can then replace $g_{j;p_{j}}$ with $g'_{j;p_{j}-1}$ at the cost of $\lt$-acting by~$a$ on all the elements of $G$ appearing in~\eqref{E:LabelingRule} to the left of $g_{j;p_{j}}$. That is, all the $g_t$ with $t \leq k_1+\cdots+k_{j-1}$ should be replaced with $g_t \lt a$, which by the definition of $d_{j;0}$ is precisely $g'_t$. Also, for $t > k_1+\cdots+k_{j-1}$ one has $g'_t=g_{t+1}$, so for all $u>j$, one has $g'_{u;p_{u}} = g_{u;p_{u}}$, or $g'_{u-1;p_{u-1}} = g_{u;p_{u}}$ if $p_j=1$.  This proves that the two labelings we are considering coincide on the edges supported in $\Delta^{k_1} \times \cdots \times \Delta^{k_{j-1}}$. Since the removal of a vertex of $\Delta^{k_j}$ has no effect on the labelings of the edges supported in $\Delta^{k_j} \times \cdots \times \Delta^{k_{\ell}}$, we are done.
\end{proof}

Note that for the extreme partitions $n=n$ and $n=1+1+\cdots +1$, we recover copies of the classical classifying spaces of the semigroup $(G,\cdot)$ and the shelf $(G,\lt)$ respectively; cf. Section~\ref{S:SimplicesCubes}.

We finish with a useful property of good $G$-labelings. Take any oriented path consisting of edges of $\Delta^{\overline{k}}$, and suppose that the edge labelings encountered along this path are $a_1,\ldots,a_p$, in this order. Consider the map $G \to G$, $g \mapsto g \lt (a_1\cdots a_p)$. This map
\begin{itemize}
\item is a shalgebra morphism due to Axioms $\Y\I$ and $\I\I\I$;
\item depends only on the source and the target vertices of the path, and not on the path itself: indeed, one can construct a homotopy between any two such paths out of triangular and square $2$-dimensional faces only, and for the two paths forming the boundary of such faces the statement follows from Axioms $\A$ (for triangles) and $\I\Y$ and $\I\I\I$ (for squares).
\end{itemize}
Thus, denoting by $V$ the set of vertices of $\Delta^{\overline{k}}$, out of any good $G$-labeling of  $\Delta^{\overline{k}}$ we get a map
\[V \times V \to \End(G).\]
This construction is useful, for instance, for understanding how to $G$-label $\Delta^{\overline{k}}\times \Delta^{\overline{m}}$ if one knows how to label $\Delta^{\overline{k}}$ and $\Delta^{\overline{m}}$. If the partition $\overline{m}$ has one part only, this is precisely our inductive labeling.

\section{Knottings of dimension $1$ and $2$}\label{S:Knottings}

The main results of this paper relate prismatic homology to certain colored knottings. In this section we
\begin{itemize}
\item explain the difference between handle-body knots and knotted trivalent graphs;
\item discuss knotted foams which are the analogues in $4$-space of knotted trivalent graphs;
\item summarize the ideas of Reidemeister/Roseman moves to knotted foams.
\end{itemize}

\subsection{Handle-body knots}

A \emph{handle-body knot (HBK)} is an embedding into $\R^3$ (or $S^3$) of a $3$-dimensional manifold whose fundamental group is free (say on $g$ generators) and whose boundary is a closed orientable genus $g$ surface. Handle-body links are defined similarly: each component has a free fundamental group and a boundary that is a closed orientable surface. We will use the abbreviation HBK for either. Any HBK can be represented by the {\it diagram of a knotted trivalent graph (KTG)}. This is a generic projection into the plane of a trivalent graph that is embedded in $3$-space and for which crossing information is indicated at transverse crossings between pairs of arcs by breaking the arc that lies below  with respect to the projection direction in the standard way. By thickening the knotted trivalent graph in the direction of the plane into which it is projected and subsequently thickening in the projection direction, an HBK is constructed. By a theorem from \cite{IshiiHKnots}, two diagrams of KTGs represent the same HBK if and only if they are related by
 a sequence of the moves depicted in Fig~\ref{Rmoves}.

Note especially, that the move $\A$ changes the underlying topology of the representative trivalent graph. For example, the typical knotting of the handcuff graph that is depicted in Fig.~\ref{handcuffs} is non-trivial as a KTG, but by applying an $\A$ move, the underlying graph changes from a handcuff graph, $0\! \! -\!\!0,$ to a theta curve, $\uptheta$.

\begin{figure}[htb]
\begin{center}
\includegraphics[width=2in]{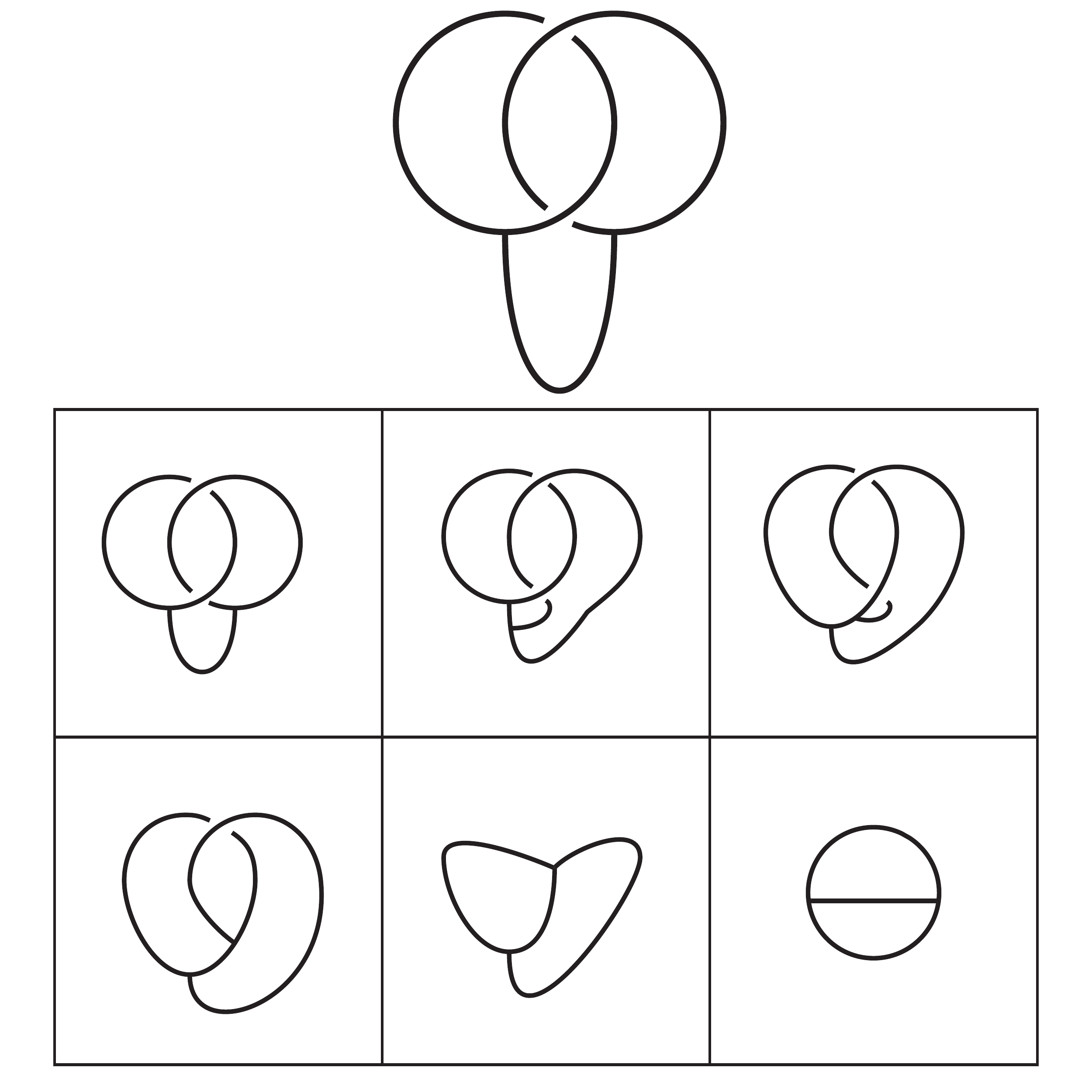}
\end{center}
\caption{A knotted trivalent graph that is unknotted as a handle-body}\label{handcuffs}
\end{figure}

\subsection{Knotted foams} In order to describe $2$-dimensional analogues of HBKs, we first describe the analogues of KTGs. Specifically, knotted foams (to be defined shortly) are to knotted surfaces as knotted trivalent graphs are to classical knots. To stretch the analogy, we consider a regular neighborhood in $3$-space of the diagram of the foam and subsequently thicken this in the direction of the projection. The resulting types of $4$-manifolds are unknown to authors. They are $4$-dimensional manifolds with boundaries that embed in $4$-space. Just as HBKs are related to cubes with knotted holes, one might presume a similar dichotomy in $4$ dimensions.

Let us define $2$-foams. Let $Y^1$ denote the graph in the triangle ($2$-simplex) $\Delta^2$ that is obtained by coning the points $(1/2,1/2,0)$,$(1/2,0,1/2)$, and $(0,1/2,1/2)$ to the barycenter $(1/3,1/3,1/3)$. For $j=0,1,2,3$, let $Y^1_j \subset \Delta^2_j= \partial_j(\Delta^2)$ denote a copy of $Y^1$ in the $j\/$th face of the tetrahedron $\Delta^3$. Any two such $Y^1_j\/$s share a boundary vertex. Let the $2$-complex $Y^2$ be obtained from  $\cup_{j=0}^3 Y^1_j$ by coning this $1$-complex to the barycenter $(1/4,1/4,1/4,1/4)$ of the tetrahedron. The space $Y^2$ (depicted in Fig.~\ref{A}) has one vertex at which four edges and six faces are incident.

\begin{figure}[htb]
\begin{center}
\includegraphics[width=.5\paperwidth]{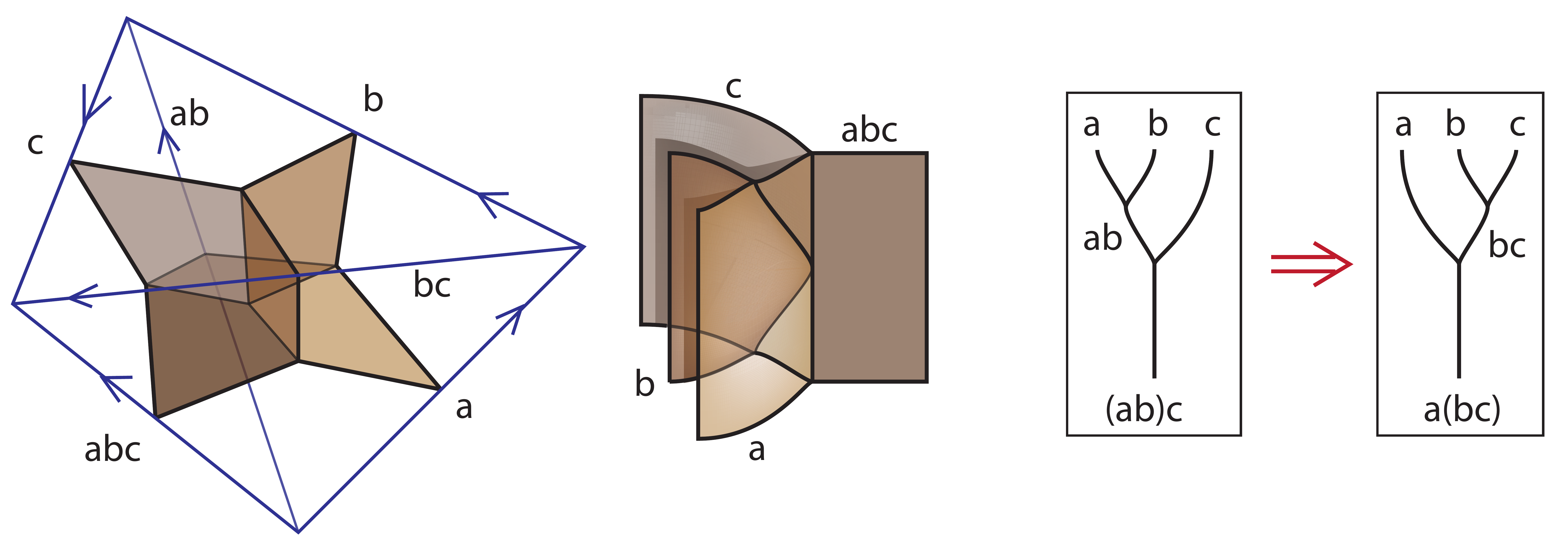}
\end{center}
\caption{The space $Y^2$ as Axiom $\A$}\label{A}
\end{figure}

A {\it $2$-foam}  (or simply a {\it foam}) is a compact orientable
 topological space $F$ for which each point has a neighborhood homeomorphic to a neighborhood of a point in $Y^2$. The boundary of a $2$-foam is a KTG. We are primarily interested in foams without boundary. There are three types of points in a $2$-foam. A {\it non-singular} point has a neighborhood that is homeomorphic to a $2$-dimensional disk. The open connected component of a non-singular point is called {\it a face}. A {\it seam} is an edge upon which three faces meet. A neighborhood of a seam point in $F$ is homeomorphic to the interior of $Y^1\times (0,1)$. A {\it vertex} is the junction of four edges and six faces; it has a neighborhood homeomorphic to the interior of $Y^2$. Of course, an edge or face might return to a given vertex, so this description is only a local picture.

A {\it knotted foam} is an embedding of a foam into $4$-space. Diagrammatic and movie techniques can be used to depict the local pictures of crossings of knotted foams. Most importantly, the moves $\A$, $\Y\I$, $\I\Y$, and $\I\I\I$ to KTGs define neighborhoods of crossings of foams. Specifically, the two sides of these moves form the boundary of the corresponding neighborhoods. Here, ``crossing'' should be interpreted in a broad sense. In the analogous situation for KTGs, ``crossings'' encompass vertices of the $Y^1\/s$ and the crossing of the form $X\/$. For knotted foams, ``crossings'' mean 
the vertices that are found at the junction of $Y^2\/$s, the crossings at which a transverse sheet passes over or under a $Y\times (0,1)$, and the triple points that occur at type $\I\I\I$ moves.  Local pictures of $3$-dimensional projections of crossings in these four situations are depicted in Figs. \ref{A}-\ref{III}. Their duality with the various prisms and their movie parametrizations are also shown.

\begin{figure}[htb]
\begin{center}
\includegraphics[width=3.6in, height=2.3in]{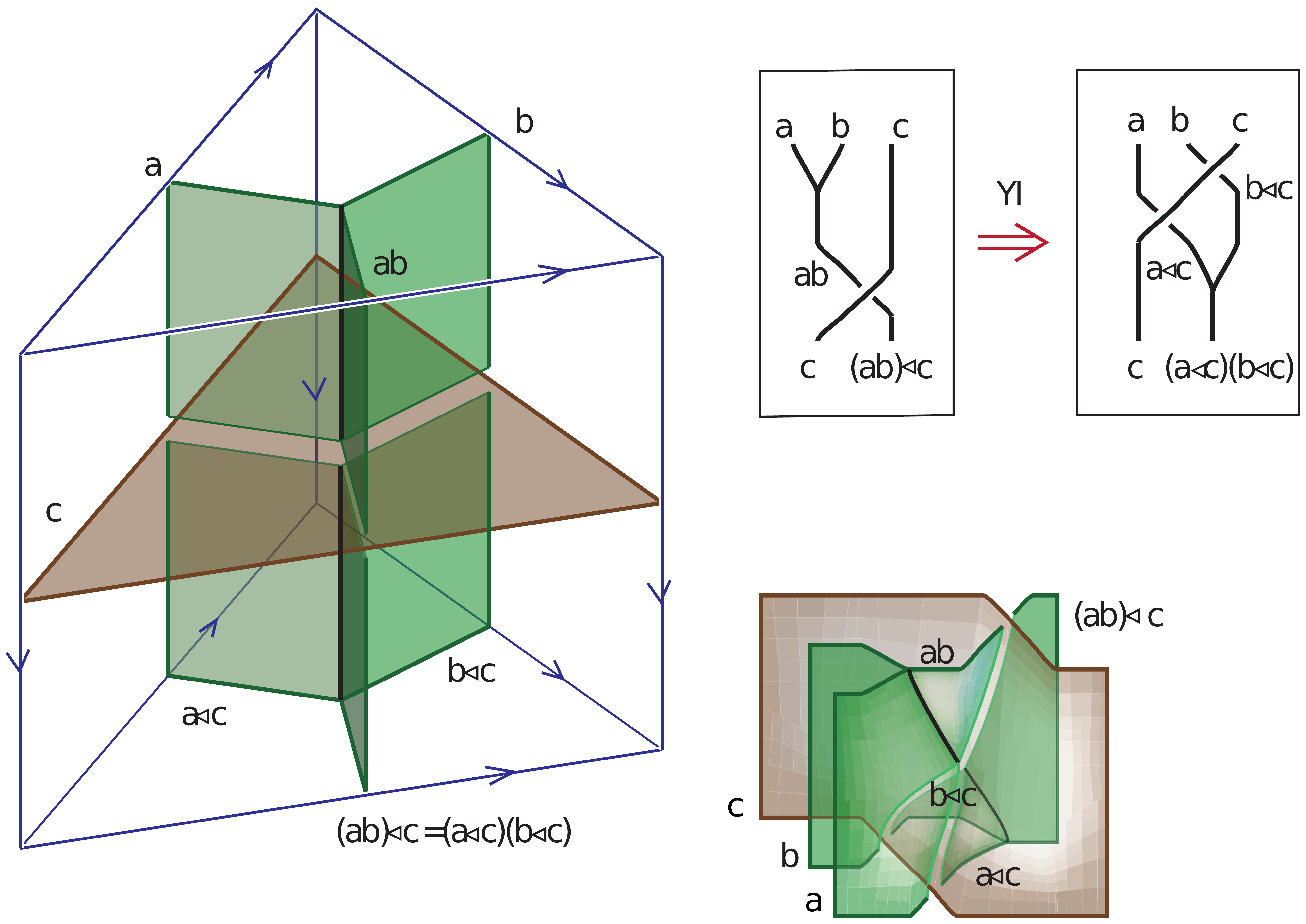}
\end{center}
\caption{The knotted foam corresponding to Axiom $\Y\I$}\label{YI}
\end{figure}

\begin{figure}[htb]
\begin{center}
\includegraphics[width=3.8in, height=2.3in]{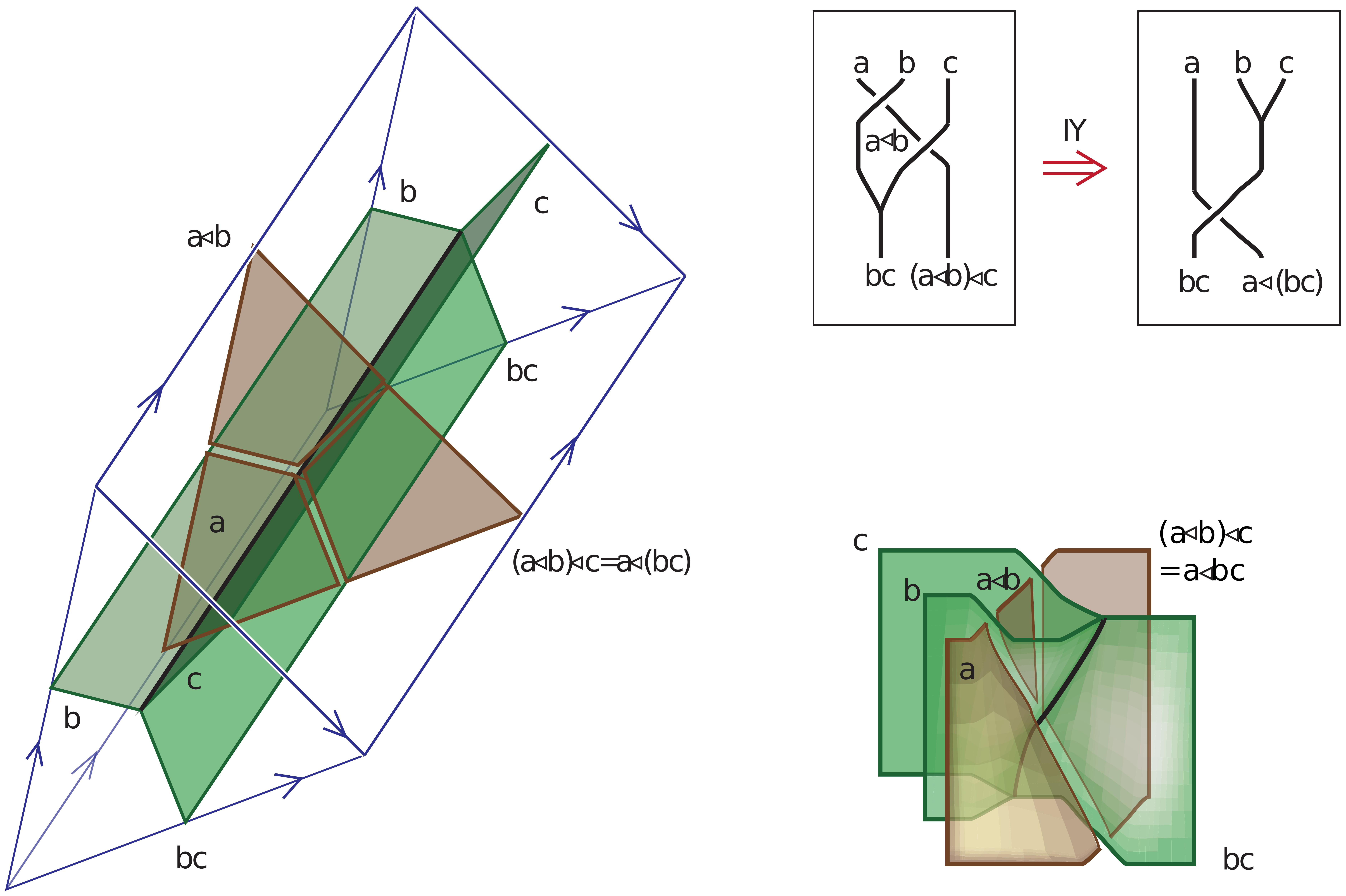}
\end{center}
\caption{The knotted foam corresponding to Axiom $\I\Y$}\label{IY}
\end{figure}

\begin{figure}[htb]
\begin{center}
\includegraphics[width=4in, height=2.3in]{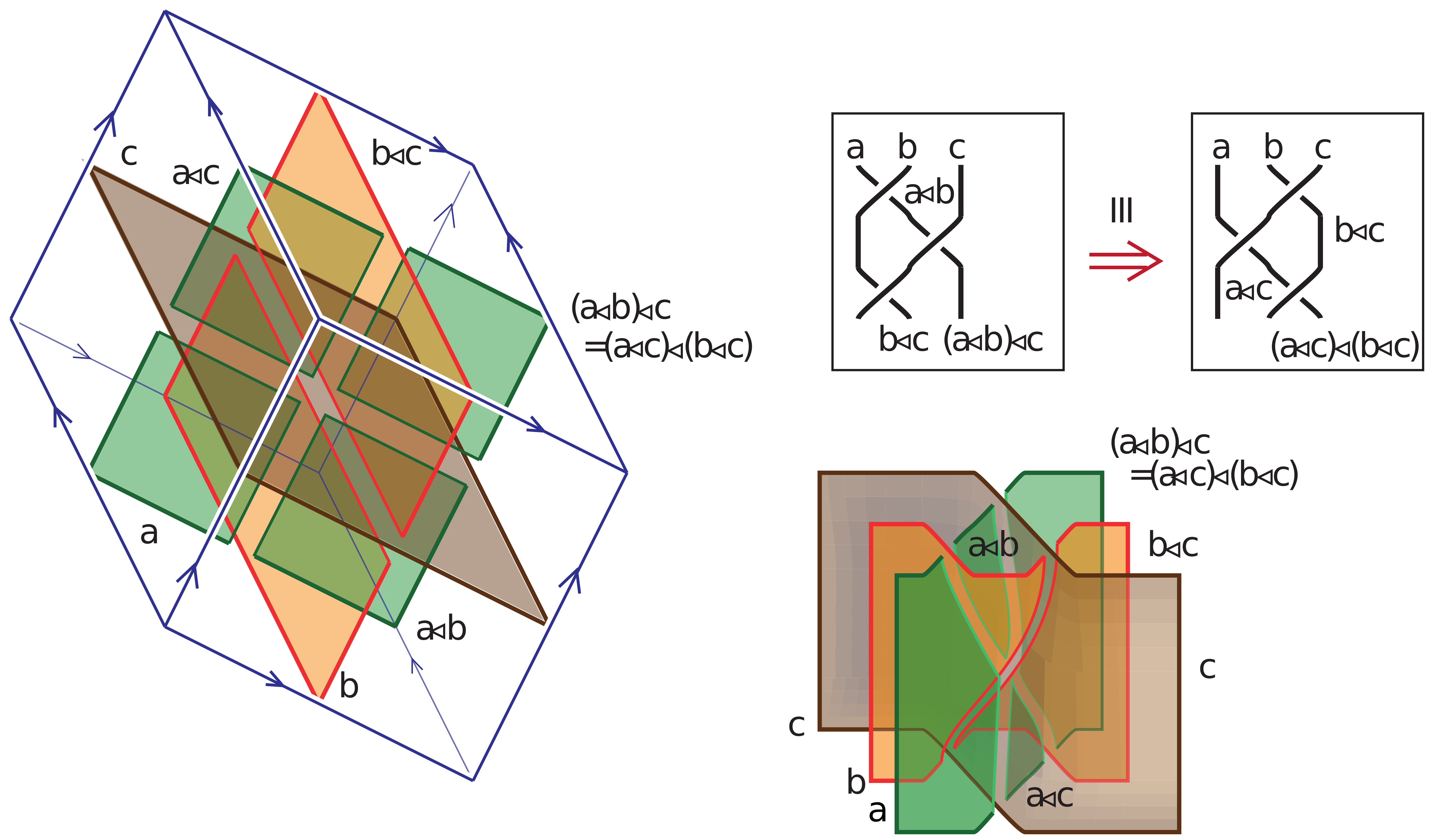}
\end{center}
\caption{The knotted foam corresponding to Axiom $\I\I\I$}\label{III}
\end{figure}

\subsection{Moves to foams}

A list of Reidemeister/Roseman moves for knotted foams in $4$-space was proposed in \cite{RR}. They are meant to be the $2$-dimensional analogue of the Reidemeister moves for KTGs.
 Unfortunately, this list is incomplete. In particular, the four versions of the twist vertex move (TVM) that are illustrated in Fig.~\ref{Twistandshout} are missing. Other moves may not have been found in \cite{RR}. Before addressing this disparity, we will briefly describe the TVMs.
  They each involve two possible ways in which a trivalent vertex can be twisted. A twist of a vertex involves two of the three legs at the vertex. The same result can be achieved by twisting another pair of legs and manipulating the vertex using Reidemeister moves of other types. The TVMs relate these two ways of twisting.
 We imagine that only two of the four moves are necessary, and the other two follow from them.

\begin{figure}[thb]
\begin{center}
\includegraphics[width=3.3in]{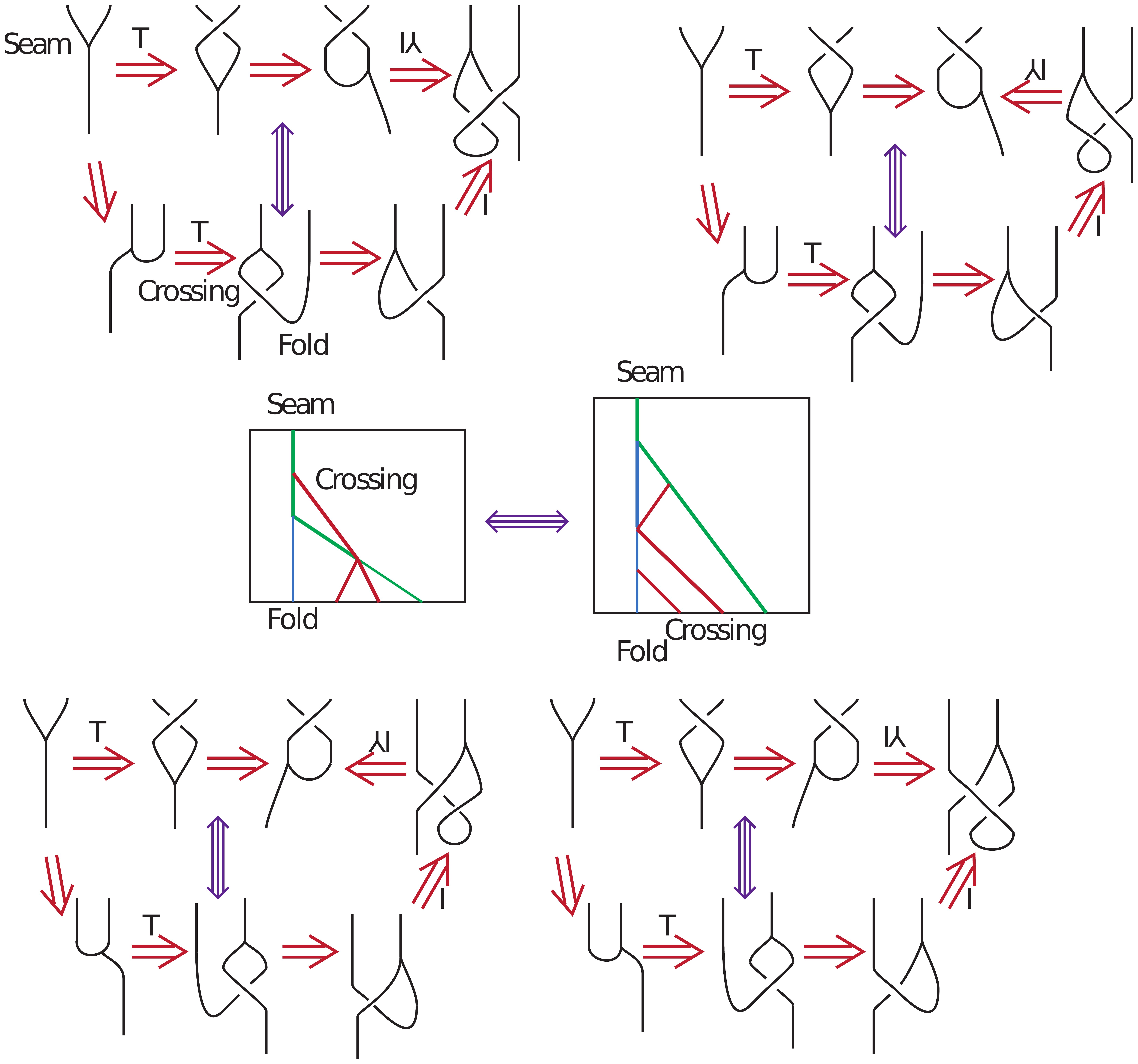}
\end{center}
\caption{The twist vertex moves for knotted foams}\label{Twistandshout}
\end{figure}

The imperfection in the result of \cite{RR} will not adversely affect our result here. Specifically, at the level of knotted foams one can restrict the notion of isotopy so that only certain moves are allowed. When a list of moves is specified, then cells are added to the
 classifying space $BG$ of a qualgebra that reflect these moves, as explained in Section~\ref{S:Degeneracies}. If additional moves are found to be necessary to fully describe isotopies of knotted foams, then additional ($4$-dimensional) cells can be added that reflect these moves.

Each qualgebra axiom corresponds to a move to the KTGs that represent HBKs. With the exception of Axiom $\A$, each describes an isotopy move for the KTG. In \cite{RR}, these isotopies are interpreted as ``atomic pieces'' of knotted foams. Many of the moves to foams can be quantified by asserting that these are strict isomorphisms. That is, each move is invertible. The invertibility of $\Y\I$, $\I\Y$, $\I\I\I$, $\I\I$, $\I$ and $\T$ account for eleven of the moves to foams
  that are given. At the level of the classifying space $BG$ that we construct, the invertibility of $\Y\I$, $\I\Y$, and $\I\I\I$ are encoded by the orientation of the corresponding prisms. The invertibility of moves $\I\I$, $\I$, and $\T$ are encoded by the orientations of some of the faces of these complexes.

Two more of the moves to foams are described in terms of naturality of crossings with respect to $2$-morphisms. These moves are described as a branch point or a twist vertex passing through a transverse sheet. They both represent degeneracies in the classifying space that we construct (Fig.~\ref{degeneratechains2} illustrates the degeneracies for TVMs).
In the case of passing a branch point through a transverse sheet, this is the move that corresponds to the declaration that the quandle chains $a|a|b$ and $a|b|b$ are degenerate. The triple point that arises in case the transverse sheet is over or under, respectively, can be removed, and the cube that is dual to the triple point has faces identified.
A similar situation happens at the twist vertex, but in this case, the faces of the induced cube and prism are degenerated.

The remaining seven moves to foams are encoded by the $4$-dimensional polytopes corresponding to  qualgebra $4$-chains, excluding the simplex $\Delta^4$ (Fig.~\ref{coloredpentagram}) which is dual to the move on the right-hand side of Fig.~\ref{MP}.

\begin{figure}[htb]
\begin{center}
\includegraphics[width=3in, height=2.4in]{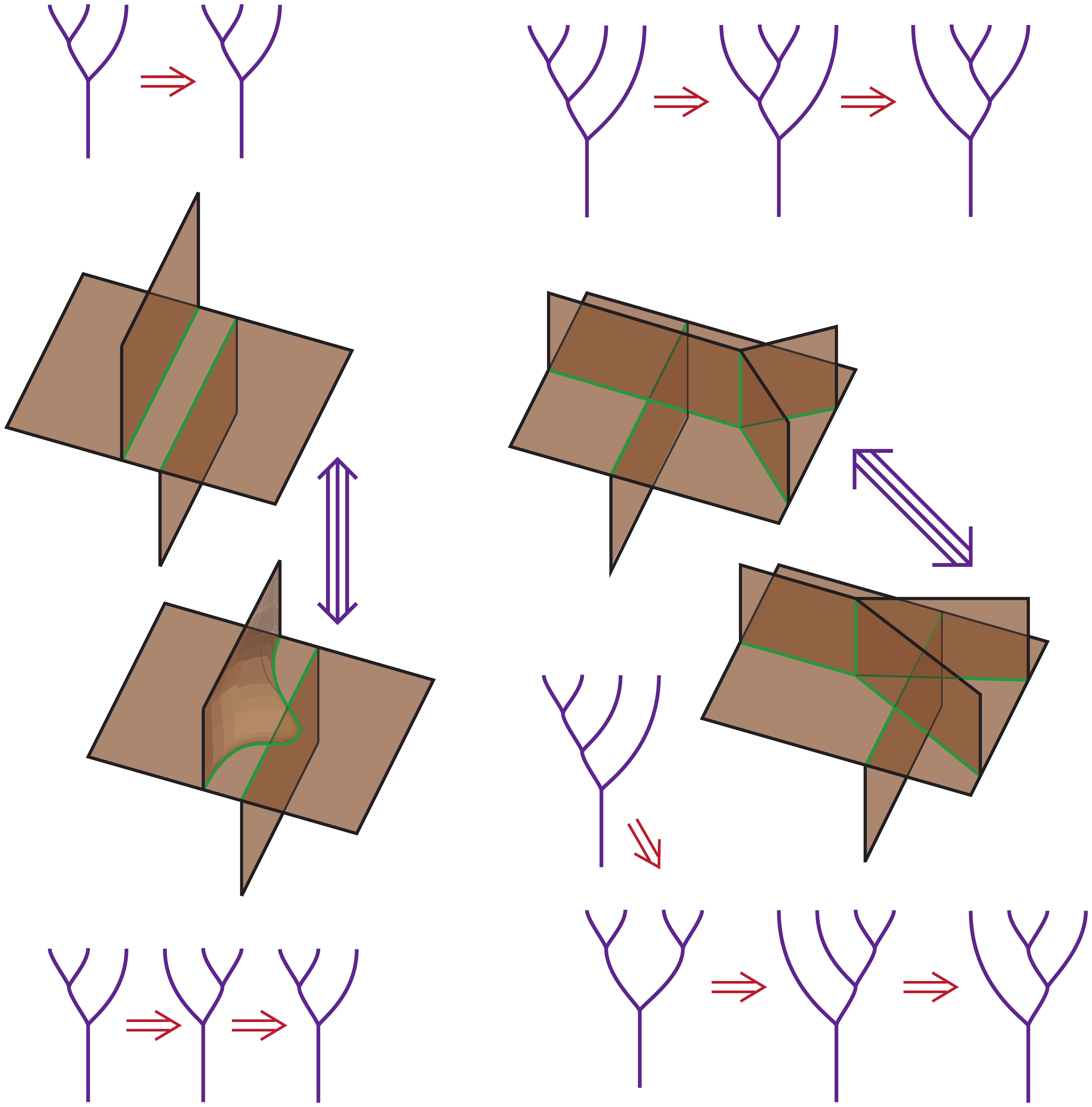}
\end{center}
\caption{The Matveev--Piergallini moves, which change the topology of the foam but preserve its $4$-manifold neighborhood}\label{MP}
\end{figure}

Recall the special $1$-dimensional move $\A$ that relates, in general, different KTGs that become isotopic after thickening, and thus represent the same HBK. Similar moves exist for $2$-dimensional knottings. Two examples are given in Fig.~\ref{MP}. They are well-known to preserve the $3$-manifold neighborhood that contains these spines as deformation retracts \cite{MatveevFoams,Piergallini}. Consequently, the $4$-dimensional neighborhoods are also preserved. Looking at the movie move versions of these two moves, one may recognize the invertibility of the associativity ``$2$-morphism'' $\A$, and what is known in several guises as the pentagon identity, the $(3,2)$ Pachner move, and the Biedenharn--Elliott identity. In our context, we will associate chains to the vertices of the figures. Recall that the associativity moves are directed; thus the left-hand figure represents the difference $(a,b,c)-(a,b,c)$, and the right-hand figure represents the boundary of a $4$-chain.

\section{Prismatic homology with degeneracies}\label{S:Degeneracies}

In \cite{Nosaka_QuSpace, Nosaka_QuSpace2, Yang_ExQuSpace}, the classifying space of a quandle $(G,\lt)$, designed to produce homotopy invariants of oriented links, was constructed by adding extra cells to the rack classifying space of $(G,\lt)$, as introduced in \cite{RackHom}. In a similar way, we will now provide a detailed description of the $4$-dimensional skeleton of a modification $\widetilde{BG}$ of the cell complex $BG$ for a qualgebra $(G,\cdot,\lt)$. In Section~\ref{S:Invariants}, it will be used to define homotopy invariants of KTGs, HBKs, and knotted $2$-foams.

Most moves to KTGs, HBKs, and knotted $2$-foams correspond to
homotopically trivial $3$- and $4$-cells. However, there are no cells that correspond to moves $\I$ and $\T$ to KTGs, and to some moves to $2$-foams (see Fig.~\ref{degeneratechains2} for example). We investigate the 
 corresponding cells that need to be added to construct homotopy invariants.

For a generator $\overline{g} \in G^{\overline{k}} \subset C_{n}$, the $n$-cell labeled by $\overline{g}$ will be denoted by $e_{\overline{g}}$. Here we use notations from Section~\ref{S:all}.

For all $a,b \in G$, the following union of $2$-cells forms a $2$-sphere:
\begin{enumerate}
  \item [(i)] $e_{(a|b)} \cup e_{(b,a \lt b)} \cup -e_{(a,b)}$
\end{enumerate}
(see Fig.~\ref{degeneratechains1}). When labeling the edges in Fig.~\ref{degeneratechains1}, we used Axiom $\T$: $ab=b(a \lt b)$. The figure also contains the corresponding Reidemeister move, in the sense to be made precise in Section~\ref{S:Invariants}. We can thus add a $3$-ball $B_{a,b}^{3}$ into the $3$-dimensional skeleton of $BG$. It is glued to the $2$-skeleton by a homeomorphism $\partial(B_{a,b}^{3}) \to e_{(a|b)} \cup e_{(b,a \lt b)} \cup -e_{(a,b)}$.

\begin{figure}[htb]
\begin{center}
\includegraphics[width=7cm]{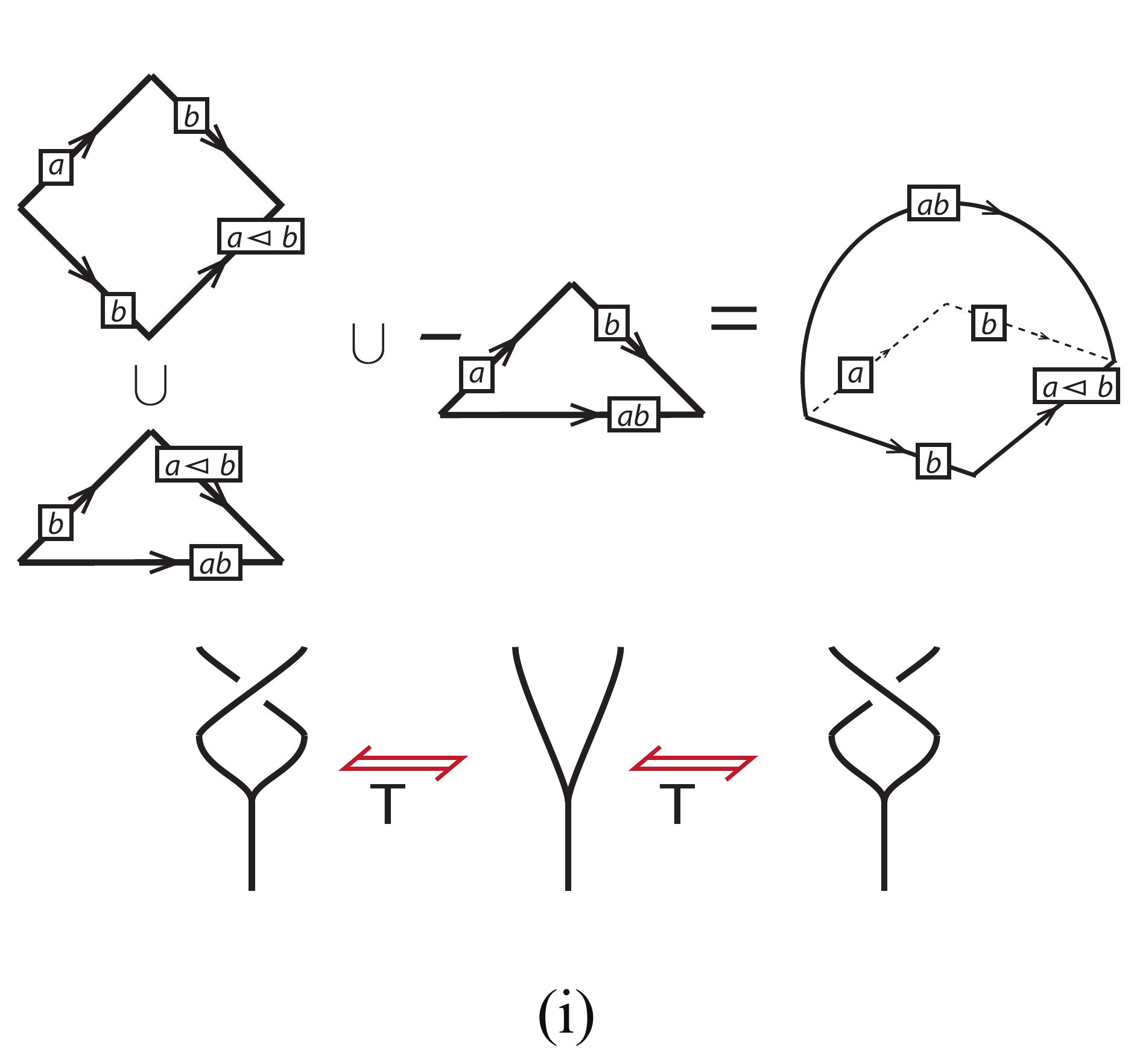}
\end{center}
\caption{The cells corresponding to the $2$-chain $(a|b)+(b,a \lt b) -(a,b)$, and the corresponding Reidemeister move}\label{degeneratechains1}
\end{figure}

Next, we add four types of extra $4$-balls. Note that for any $a,b,c \in G$, the following unions of $3$-cells form $3$-spheres (see Fig.~\ref{degeneratechains2}):
\begin{enumerate}
  \item [(ii)] $S_{1}^{3}:= e_{((a,b)|b)} \cup B_{a,b}^{3} \cup -B_{(ab)^{-1},b}^{3} \cup -B_{b,b}^{3}$,
  \item [(iii)] $S_{2}^{3}:= e_{(a|(a,b))} \cup B_{a,b}^{3} \cup -B_{a \lt b, (ab)^{-1}}^{3} \cup -B_{a,a}^{3}$,
  \item [(iv)] $S_{3}^{3}:= e_{(a|b|c)} \cup e_{(a|(c, b \lt c))} \cup -e_{(a|(b,c))}$,
  \item [(v)] $S_{4}^{3}:= e_{(a|b|c)} \cup e_{((b,a \lt b)|c)} \cup -B_{a \lt c,b \lt c}^{3} \cup -e_{((a,b)|c)} \cup B_{a,b}^{3}$.
\end{enumerate}
When the labels need to be specified, we use the notation $S_{i;a,b,c}^{3}$ instead of $S_{i}^{3}$. To simplify presentation, here we suppose that our qualgebra $G$ is in fact a group. The case of a general qualgebra can be settled by a meticulous treatment of orientations.

\begin{figure}[htb]
\begin{center}
\includegraphics[width=.6\paperwidth]{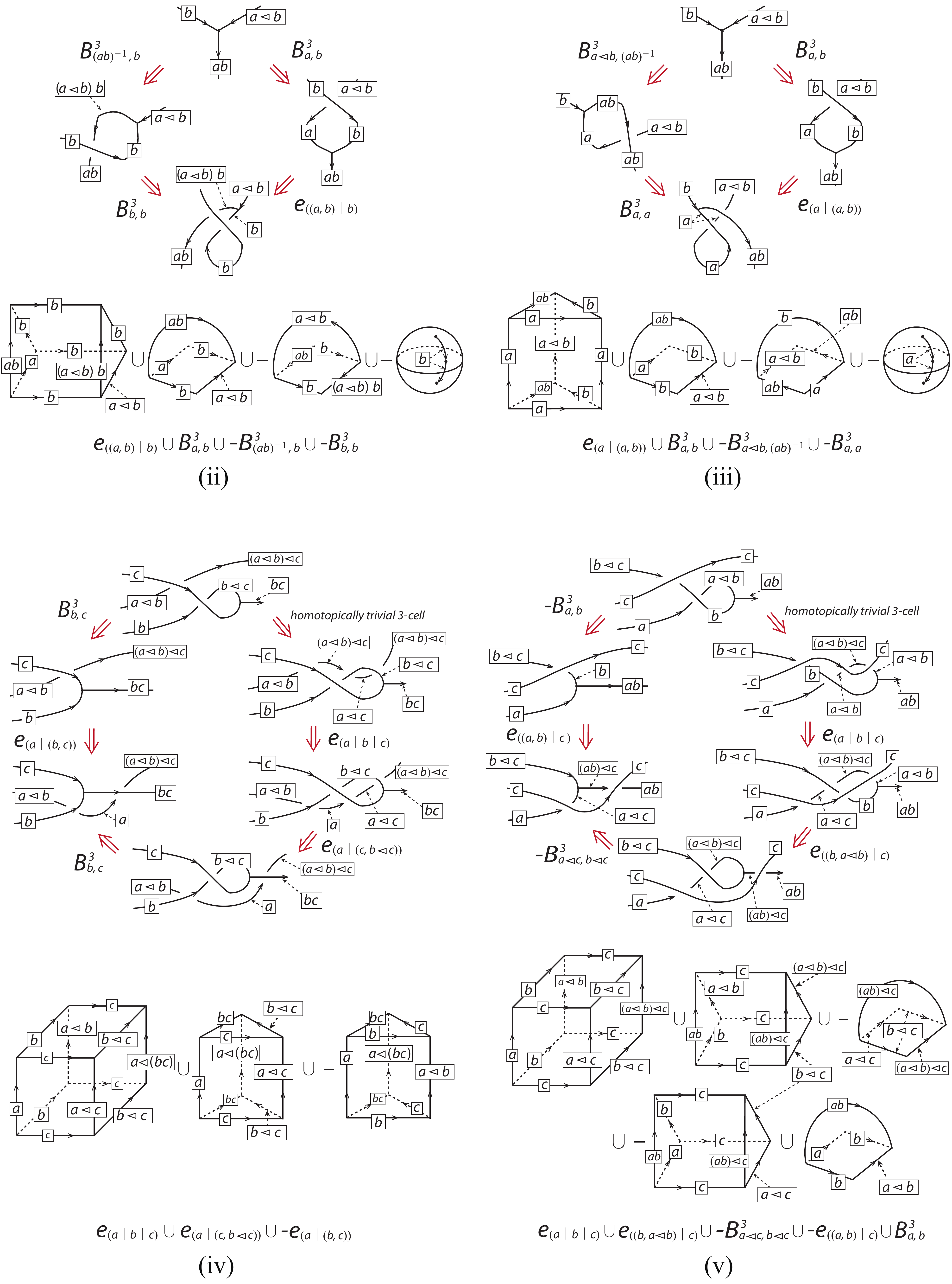}
\end{center}
\caption{Degenerate $3$-chains and their corresponding foam moves}\label{degeneratechains2}
\end{figure}

For every sphere $S_{i;a,b,c}^{3}$, glue a $4$-ball $B_{i;a,b,c}^{4}$ via a homeomorphism $\partial(B_{i;a,b,c}^{4}) \rightarrow S_{i;a,b,c}^{3}$. The cell complex $\widetilde{BG}$ obtained from $BG$ by adding all the $3$-balls $B_{a,b}^{3}$ and all the $4$-balls $B_{i;a,b,c}^{4}$ is called the \emph{classifying space of the qualgebra $(G,\cdot,\lt)$}. The corresponding homology groups are denoted by $H_n^{\rm Q \rm A}(G)$, and called the \emph{qualgebra homology} groups of $G$.

\section{Homological invariants of knottings}\label{S:Invariants}

This section contains our main results. Roughly, Theorem~\ref{main1} states: given a qualgebra $G$, a $G$-colored knotted object in $\R^n$ for $n=2$ or $3$ represents a cycle in $Z_n^{\rm Q \rm A}(G)$. When two diagrams represent equivalent knottings, the representative cycles are homologous. Further, in Theorem~\ref{main2} we interpret colored diagrams of knottings as based homotopy classes of maps from $S^n$ into the qualgebra classifying space $\widetilde{BG}$. Here the base point is the point at infinity in the one point compactification of $\R^n$.

A {\it qualgebra coloring} of a normally oriented diagram of a KTG (resp. knotted foam) is an assignment of elements to the normally oriented arcs (resp. faces) of the diagram that satisfies the conditions that are indicated in Fig.~\ref{coloredvertices} (resp. Figs. \ref{A}-\ref{III}).
 In the Fig.~\ref{coloredvertices} normal orientations are indicated explicitly. The oriented colored $Y^1$ represents a 
 positive 
 chain $(a,b)$. 
 The colored crossing in the middle of the figure is negative; the crossing on the right is positive.
 By convention, a positive chain has two arrows pointing clockwise and one pointing counter clockwise. In the Figs.~\ref{A} through \ref{III} the orientations are suggested by the orientations of the edges of the $3$-dimensional solids that encompass the crossings of the foams. We require that the normal orientations for $Y^1$ are never cyclic. KTGs and knotted foams always admit such orientations.

For convenience, let us call an HBK
or a knotted foam a {\it knotted object}. A generic projection (with crossing information indicated via broken arcs or broken surfaces) of a KTG representing the HBK to the plane, or of a knotted foam to the $3$-space,
will be called a {\it knotted object diagram (KOD)}.  When $G$ is a qualgebra, a $G$-colored KOD represents a cycle in $Z^{\rm P}_n(G)$ for $n=2$ or $3$. In fact, such a representation is given in any dimension. The representation goes as follows:
\begin{itemize}
\item  a colored $Y^1$ (resp. $Y^2$) represents a chain of the form $\pm (a,b)$
(resp, $ \pm (a,b,c)$);
\item a colored (classical) crossing $X$ determines a chain $\pm a|b$ (see Fig.~\ref{coloredvertices});
\item a colored $\Y\I$-type crossing represents a chain $\pm(a,b)|c$;
\item a colored $\I \Y$-type crossing represents a chain $\pm a|(b,c)$;
\item a colored $\I\I\I$-type crossing represents a chain $\pm a|b|c$.
\end{itemize}
Signs are determined by the orientations.
 A closed colored KOD then represents a cycle that is the signed sum of the chains represented by all its (generalized) crossings.

\begin{theorem} \label{main1} The qualgebra homology class determined by the represented cycle of a closed KOD is independent of the representative diagram.
\end{theorem}

\begin{proof}
 If different representative diagrams of a knotted object are chosen, they differ by a finite collection of Reidemeister or Roseman-type moves. Going through the list of these moves, one verifies that each move either does not affect the represented cycle, or changes its qualgebra homology class by a boundary.
\end{proof}

\begin{theorem}\label{main2} For $n=2$ and $3$, colored equivalence classes of KODs correspond to based homotopy classes of maps $S^n \rightarrow \widetilde{BG}$. Here $n=2$ in the case of an HBK, and $n=3$ in the case of a knotted foam. \end{theorem}

\begin{proof}
 First, we must disambiguate the notion of equivalence. In the case of HBKs, the diagrams represent the same HBK if and only if they are related by a finite sequence of moves taken from Fig.~\ref{Rmoves}. In the case of knotted foams, we consider knotted foam diagrams modulo the moves given in ~\cite{RR} and the moves given in Figs.~\ref{Twistandshout} and~\ref{MP}.

Next, degeneracies are described in Section~\ref{S:Degeneracies} for both the knotted foams and the handle-body knots. In that section, extra cells of dimension $3$ and $4$ are described as balls whose boundaries have a particular decomposition as squares, triangles, tetrahedra, prisms, and cubes. In the case for which the $3$-dimensional boundary degenerates, these cells are described in terms of the lower dimensional moves.

\begin{figure}[htb]
\begin{center}
\includegraphics[width=10cm]{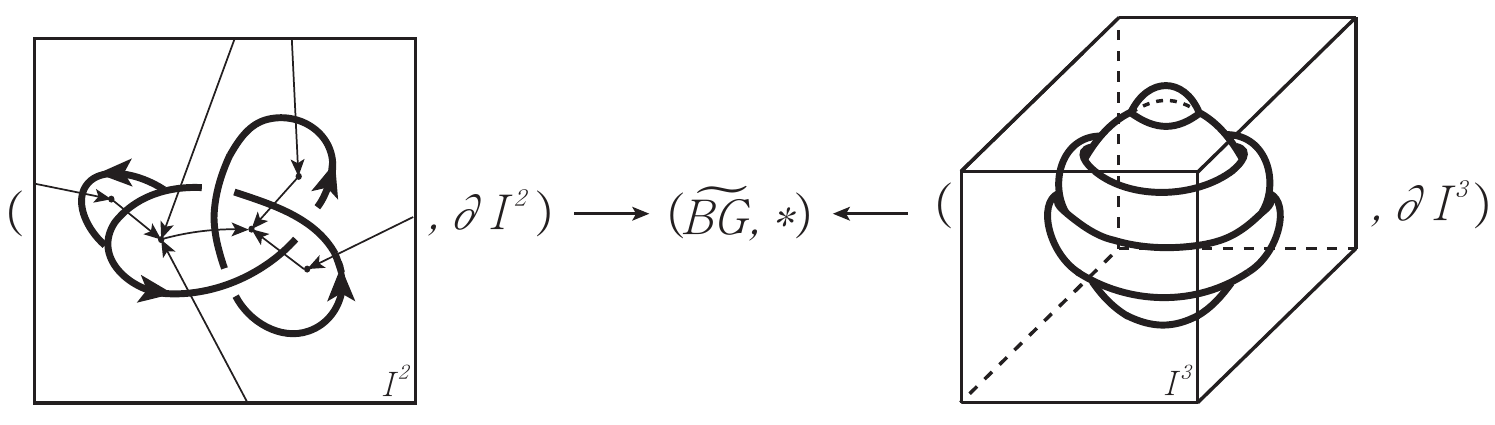}
\end{center}
\caption{Defining the homotopy class of the map}\label{Homotopyinvariant}
\end{figure}

Finally, we comment on a schematic diagram  in Fig.~\ref{Homotopyinvariant} that indicates the gist of the construction. The KOD is placed inside a large rectangle or rectangular box. Each colored crossing (in case $n=2$, $\Y$  or $\X$; in case $n=3$, $\A,$ $\Y\I$, $\I\Y$, or $\I\I\I$) is enveloped in a generalized prism. The union of the prisms is mapped to the classifying space, and the boundary rectangle or cube is mapped to the base point as is the space outside the knotting. Any diagrammatic equivalence corresponds to a higher dimensional cell in the classifying space $\widetilde{BG}$. \end{proof}

\section*{Acknowledgements}  JSC would like to thank Masahico Saito, Atsushi Ishii, and Jamie Vicary for valuable conversations about some aspects of this work.
All the authors are grateful to Petr Vojt\v{e}chovsk\'{y} and Patrick Dehornoy for their dedicating time at the $4$\/th Mile High Conference on Nonassociative Mathematics towards the discussion of self-distributive strutures.

\bigskip
\footnotesize
\bibliographystyle{alpha}
\bibliography{biblio}

\newcommand{\etalchar}[1]{$^{#1}$}
\def\cprime{$'$}
\begin{thebibliography}{{Kau}17}

\bibitem[BH81]{BH_cubes}
Ronald Brown and Philip~J. Higgins.
\newblock On the algebra of cubes.
\newblock {\em J. Pure Appl. Algebra}, 21(3):233--260, 1981.

\bibitem[Car15]{RR}
J.~Scott Carter.
\newblock Reidemeister/{R}oseman-type moves to embedded foams in 4-dimensional
  space.
\newblock In {\em New ideas in low dimensional topology}, volume~56 of {\em
  Ser. Knots Everything}, pages 1--30. World Sci. Publ., Hackensack, NJ, 2015.

\bibitem[CES04]{HomologyYB}
J.~Scott Carter, Mohamed Elhamdadi, and Masahico Saito.
\newblock Homology theory for the set-theoretic {Y}ang--{B}axter equation and
  knot invariants from generalizations of quandles.
\newblock {\em Fund. Math.}, 184:31--54, 2004.

\bibitem[CIST17]{CIST:HomMCQ}
Scott Carter, J.~Scott, Atsushi Ishii, Masahico Saito, and Kokoro Tanaka.
\newblock Homology for quandles with partial group operations.
\newblock {\em Pacific J. Math.}, 287(1):19--48, 2017.

\bibitem[CJK{\etalchar{+}}03]{QuandleHom}
J.~Scott Carter, Daniel Jelsovsky, Seiichi Kamada, Laurel Langford, and
  Masahico Saito.
\newblock Quandle cohomology and state-sum invariants of knotted curves and
  surfaces.
\newblock {\em Trans. Amer. Math. Soc.}, 355(10):3947--3989, 2003.

\bibitem[CS17]{EK_MS:IdentitiesHomology}
W.~Edwin Clark and Masahico Saito.
\newblock Quandle identities and homology.
\newblock In {\em Knots, links, spatial graphs, and algebraic invariants},
  volume 689 of {\em Contemp. Math.}, pages 23--35. Amer. Math. Soc.,
  Providence, RI, 2017.

\bibitem[Deh00]{Dehornoy2}
Patrick Dehornoy.
\newblock {\em Braids and self-distributivity}, volume 192 of {\em Progress in
  Mathematics}.
\newblock Birkh\"auser Verlag, Basel, 2000.

\bibitem[Deh06]{DehornoyParBr}
Patrick Dehornoy.
\newblock The group of parenthesized braids.
\newblock {\em Adv. Math.}, 205(2):354--409, 2006.

\bibitem[Deh07]{DehornoyFreeALDS}
Patrick Dehornoy.
\newblock Free augmented {LD}-systems.
\newblock {\em J. Algebra Appl.}, 6(1):173--187, 2007.

\bibitem[Dr{\'a}95]{DrapalLDM}
Ale{\v{s}} Dr{\'a}pal.
\newblock On the semigroup structure of cyclic left distributive algebras.
\newblock {\em Semigroup Forum}, 51(1):23--30, 1995.

\bibitem[Dr{\'a}97]{DrapalLDM2}
Ale{\v{s}} Dr{\'a}pal.
\newblock Finite left distributive algebras with one generator.
\newblock {\em J. Pure Appl. Algebra}, 121(3):233--251, 1997.

\bibitem[FRS95]{RackHom}
Roger Fenn, Colin Rourke, and Brian Sanderson.
\newblock Trunks and classifying spaces.
\newblock {\em Appl. Categ. Structures}, 3(4):321--356, 1995.

\bibitem[GS83]{SimplIsHoch}
Murray Gerstenhaber and Samuel~D. Schack.
\newblock Simplicial cohomology is {H}ochschild cohomology.
\newblock {\em J. Pure Appl. Algebra}, 30(2):143--156, 1983.

\bibitem[Ish08]{IshiiHKnots}
Atsushi Ishii.
\newblock Moves and invariants for knotted handlebodies.
\newblock {\em Algebr. Geom. Topol.}, 8(3):1403--1418, 2008.

\bibitem[Ish15]{AI:MCQ}
Atsushi Ishii.
\newblock A multiple conjugation quandle and handlebody-knots.
\newblock {\em Topology Appl.}, 196(part B):492--500, 2015.

\bibitem[{Kau}17]{LK:SimplKho}
L.~H {Kauffman}.
\newblock {Simplicial Homotopy Theory, Link Homology and Khovanov Homology}.
\newblock {\em ArXiv e-prints}, January 2017.

\bibitem[{Leb}14]{VL:BranchedBraids}
Victoria {Lebed}.
\newblock {Knotted $3$-valent graphs, branched braids, and
  multiplication-conjugation relations in a group}.
\newblock {\em Proc. of Intelligence of Low-Dimensional Topology}, pages
  86--100, 2014.

\bibitem[Leb15]{VL:qual}
Victoria Lebed.
\newblock Qualgebras and knotted 3-valent graphs.
\newblock {\em Fund. Math.}, 230(2):167--204, 2015.

\bibitem[Leb17]{VL:Systems}
Victoria Lebed.
\newblock {Braided Systems: a Unified Treatment of Algebraic Structures with
  Several Operations}.
\newblock {\em Homology Homotopy Appl.}, 19(2):141--174, 2017.

\bibitem[Lod92]{Cyclic}
Jean-Louis Loday.
\newblock {\em Cyclic homology}, volume 301 of {\em Grundlehren der
  Mathematischen Wissenschaften [Fundamental Principles of Mathematical
  Sciences]}.
\newblock Springer-Verlag, Berlin, 1992.
\newblock Appendix E by Mar{\'{\i}}a O. Ronco.

\bibitem[LV17]{LebedVendramin}
Victoria Lebed and Leandro Vendramin.
\newblock Homology of left non-degenerate set-theoretic solutions to the
  {Y}ang--{B}axter equation.
\newblock {\em Adv. Math.}, 304:1219--1261, 2017.

\bibitem[Mat87]{MatveevFoams}
S.~V. Matveev.
\newblock Transformations of special spines, and the {Z}eeman conjecture.
\newblock {\em Izv. Akad. Nauk SSSR Ser. Mat.}, 51(5):1104--1116, 1119, 1987.

\bibitem[Nos11]{Nosaka_QuSpace}
Takefumi Nosaka.
\newblock On homotopy groups of quandle spaces and the quandle homotopy
  invariant of links.
\newblock {\em Topology Appl.}, 158(8):996--1011, 2011.

\bibitem[Nos13]{Nosaka_QuSpace2}
Takefumi Nosaka.
\newblock Quandle homotopy invariants of knotted surfaces.
\newblock {\em Math. Z.}, 274(1):341--365, 2013.

\bibitem[Pie88]{Piergallini}
Riccardo Piergallini.
\newblock Standard moves for standard polyhedra and spines.
\newblock {\em Rend. Circ. Mat. Palermo (2) Suppl.}, (18):391--414, 1988.
\newblock Third National Conference on Topology (Italian) (Trieste, 1986).

\bibitem[Rie08]{Emily}
Emily Riehl.
\newblock A leisurely introduction to simplicial sets.
\newblock {www.math.jhu.edu/~eriehl/ssets.pdf}, 2008.

\bibitem[Yan17]{Yang_ExQuSpace}
Seung~Yeop Yang.
\newblock Extended quandle spaces and shadow homotopy invariants of classical
  links.
\newblock {\em J. Knot Theory Ramifications}, 26(3):1741010, 2017.

\end{thebibliography}
\end{document}